\documentclass[11pt]{article}

\usepackage{amsmath,amsthm,verbatim,amssymb,amsfonts,amscd, graphicx}
\usepackage{graphics}
\usepackage{cite}
\usepackage{xcolor}
\usepackage{epstopdf}
\usepackage{float}
\usepackage{hyperref}
\usepackage{mathtools}
\usepackage{cancel}
\usepackage{mathrsfs}
\DeclareMathAlphabet{\mathantt}{OT1}{antt}{li}{it}
\DeclareMathAlphabet{\mathpzc}{OT1}{pzc}{m}{it}

\theoremstyle{plain}
\newtheorem{theorem}{Theorem}

\newtheorem{lemma}{Lemma}
\newtheorem{proposition}{Proposition}

\theoremstyle{definition}

\newtheorem*{remark}{Remark}
\DeclareMathOperator*{\argmin}{argmin}
\DeclareMathOperator{\prox}{prox}
\DeclareMathOperator{\sgn}{sgn}
\newcommand{\dup}{\overline{D}_{p}} 
\newcommand{\dlo}{\underline{D}_{p}} 

\newcommand{\xad}{{x_{\alpha}^\delta}}
\newcommand{\xadii}{{x_{\alpha,\delta}^{II}}}
\newcommand{\xadn}{{\mathpzc{x}_\alpha^\delta}}
\newcommand{\xan}{{\mathpzc{x}_\alpha}}
\newcommand{\etan}{\eta_n}
\newcommand{\xadiin}{{\mathpzc{x}^{II}_{\alpha,\delta}}}
\newcommand{\pad}{{p_{\alpha}^\delta}}
\newcommand{\padii}{{p_{\alpha,\delta}^{II}}}

\newcommand{\pa}{{p_{\alpha}}}
\newcommand{\yd}{{y^\delta}}
\newcommand{\ydn}{\mathpzc{y}^\delta}
\newcommand{\yn}{\mathpzc{y}}
\newcommand{\hq}[1]{h_{#1}}
\newcommand{\hqi}[1]{h_{#1}^{-1}}
\newcommand{\Dp}{{\Delta p_\alpha}}
\newcommand{\Dy}{{\Delta y}}
\newcommand{\Dpn}{{\Delta\mathpzc{p}_\alpha}}
\newcommand{\Dyn}{{\Delta\mathpzc{y}}}
\newcommand{\padn}{{\mathpzc{p}^\delta_\alpha}}
\newcommand{\padiin}{{\mathpzc{p}^{II}_{\alpha,\delta}}}
\newcommand{\pan}{{\mathpzc{p}_{\alpha}}}
\newcommand{\paiin}{{\mathpzc{p}^{II}_\alpha}}
\newcommand{\Dpii}{\Delta p^{II}_\alpha}
\newcommand{\Dpiiast}{\Delta p^{II}_{\alpha_\ast}}
\newcommand{\Dpiin}{\Delta\mathpzc{p}^{II}_\alpha}
\newcommand{\ns}{{n^*}}
\newcommand{\realp}[2]{{#1 #2}}

\let\oldref\ref
\renewcommand{\ref}[1]{(\oldref{#1})}

\providecommand{\keywords}[1]{\textbf{Keywords:} #1}

\newcounter{example}[section]

\author{Stefan Kindermann\footnote{Johannes Kepler University Linz, Industrial Mathematics Institute, Altenbergerstra{\ss}e 69,
A-4040 Linz, Austria (kindermann@indmath.uni-linz.ac.at)} , Kemal Raik\footnote{Johannes Kepler University Linz, Industrial Mathematics Institute, Altenbergerstra{\ss}e 69,
A-4040 Linz, Austria (kemal.raik@indmath.uni-linz.ac.at) }}	

\title{Convergence of Heuristic Parameter Choice Rules for Convex Tikhonov Regularisation}

\begin{document}
	\maketitle
	\begin{abstract}
	We investigate the convergence theory of several known as well as new heuristic parameter choice rules for convex Tikhonov regularisation. The success of such methods is dependent on whether certain restrictions on the noise are satisfied. In the linear theory, such conditions are well understood and hold for typically irregular noise. In this paper, we extend the convergence analysis of heuristic rules using noise restrictions to the convex setting and prove convergence of the aforementioned methods therewith. The convergence theory is exemplified for the case of an ill-posed problem with a diagonal forward operator in $\ell^q$ spaces. Numerical examples also provide further insight.
	\end{abstract}
	
	\keywords{ill-posed problems, convex regularisation, heuristic parameter choice rules}
	
\section{Introduction}	
	
Let $X$ and $Y$ be Banach and Hilbert spaces, respectively. We consider the ill-posed problem
\begin{equation*}
Ax=y, 
%\label{illposed}
\end{equation*}
where $A:X\to Y$ is a continuous linear operator and only noisy data $y^\delta=y+e$ is available and $\delta$ such that $\|e\|\le\delta$ is defined to be the noise level. In other words, we assume that the data does not depend continuously on the solution. We therefore determine a regularised solution \`{a} la Tikhonov:
\begin{equation*}
x^\delta_\alpha\in\argmin_{x\in X}\mathcal{T}^\delta_\alpha(x),
\end{equation*}
where
\begin{equation}\label{tikfun}
\begin{aligned}
\mathcal{T}^\delta_\alpha(x):=\frac{1}{2}\|Ax-y^\delta\|^2+\alpha\mathcal{R}(x),
\end{aligned}
\end{equation}
is the \emph{Tikhonov functional}, and the regularisation term given by the functional $\mathcal{R}:X\to \mathbb{R}\cup\{\infty\}$ is assumed to be convex, proper, coercive, weak-$\ast$ lower semicontinuous and $\alpha\in(0,\alpha_\text{max})$ is the so-called \emph{regularisation parameter}. The aforementioned properties ensure the existence of a minimiser for $\mathcal{T}^\delta_\alpha$ (cf.~\cite{variationalimaging}). In this way, we seek to approximate an \emph{$\mathcal{R}$-minimising solution}, $x^\dagger\in X$ (cf.~\cite{variationalimaging}).

The choice of regularisation parameter is pivotal for any reasonable approximation of $x^\dagger$. There are several classes of rules which select this parameter
(cf., e.g., \cite{hamarikpalmraus}), the  best-known being the a-posteriori rules which select $\alpha_\ast:=\alpha(\delta,y^\delta)$ in dependence of the noise level and the measured noisy data. A classic example is Morozov's discrepancy principle (cf.~\cite{discrepancy}). The drawback is that in practical situations, the noise level is usually unknown. In this case, one may opt to use another class of parameter choice rules, namely the so-called \emph{heuristic} rules, which select $\alpha_\ast:=\alpha(y^\delta)$ in dependence of the measured data alone (i.e., without knowledge of the noise level). Their pitfall, however, comes in the form of the Bakushinskii veto (cf.~\cite{veto}), which essentially asserts that a heuristic rule cannot yield a convergent regularisation method in the worst case scenario. However, it was proven recently  in 
\cite{kinderquasi,kinderabstract} 
for linear regularisation methods
that certain heuristic rules yield a 
convergent  method if some \emph{noise conditions} are postulated.
It is important to note that the aforementioned noise conditions utilised the spectral theory for self-adjoint linear operators. A recent discussion and extension of the noise conditions within the linear theory may be found in \cite{raikweaklybounded}. 
%and such tools are not available in our setting. 

The topic of this paper is the corresponding analysis for the convex case. Note that the tools from the linear theory are no longer applicable due to the absence of spectral theory. For the mentioned setting, some heuristic parameter choice rules were considered in the literature: B. Jin and Lorenz in \cite{jinlorenz} discussed the heuristic discrepancy rule and a version of the discrete quasi-optimality rule. The heuristic discrepancy rule was also considered for the augmented Lagrangian method and Bregman iteration for nonlinear operators in the work of Q. Jin (c.f.~\cite{QJinHankeRausBanach,QJinLagrangian,QJinBregIt}). A numerical study of certain heuristic rules was  investigated in \cite{kindermannmutimbu}.

Crucial to the convergence theory of heuristic rules are restrictions on the noise. In the linear theory, they take the form of Muckenhoupt-type conditions. In the convex case, some rather abstract conditions were proposed in \cite{jinlorenz,QJinHankeRausBanach}. However, the validity of these conditions remains unclear.

In this paper, we propose several heuristic rules and, as  main contribution, provide a convergence analysis by postulating so-called auto-regularisation conditions. They reduce to Muckenhoupt-type conditions in the setting of $\ell^q$ regularisation with a diagonal operator $A$, allowing us to subsequently investigate their validity for typical cases. The main results 
are Theorems~\ref{th:two}, \ref{th:three}, \ref{th:four} in Section~\ref{sec:two}
(with abstract conditions), and Theorems~\ref{th:six}, \ref{th:seven}, \ref{th:eight} (specific convergence conditions for the diagonal case) in Section~\ref{DiagSec}. 
Furthermore, we provide a detailed 
numerical case study of these heuristic methods in Section~\ref{sec:four}.

\section{Heuristic parameter choice rules}\label{sec:two}
The heuristic rules we consider select the parameter $\alpha $
in the Tikhonov functional \eqref{tikfun} 
\[
\alpha= \alpha_\ast\in\argmin_{\alpha\in(0,\alpha_\text{max})}\psi(\alpha,y^\delta),
\]
as the global minimiser of a functional
\[
\psi:(0,\alpha_\text{max})\times Y\to\mathbb{R}.%\cup\{0\}.
\]
We investigate the following four functionals:
\begin{alignat*}{2}
    \psi_{\text{HD}}(\alpha,y^\delta)&:=\frac{1}{\alpha}\|\yd-A\xad\|^2,&&\text{the heuristic discrepancy rule (HD)},
    \\
    \psi_{\text{HR}}(\alpha,y^\delta)&:=\frac{1}{\alpha}\langle y^\delta-A\xadii,y^\delta-A\xad\rangle,&
    \qquad&\text{the Hanke-Raus rule (HR)},
    \\
    \psi_{\text{SQO}}(\alpha,y^\delta)&:=D^{\text{sym}}_{\xi^{II}_{\alpha,\delta},\xi^\delta_\alpha}(x^{II}_{\alpha,\delta},x^\delta_\alpha),&&\text{the symmetric quasi-optimality rule (sQO)},
    \\
    \psi_{\text{RQO}}(\alpha,y^\delta)&:=D_{\xi^\delta_\alpha}(x^{II}_{\alpha,\delta},x^\delta_\alpha),&&\text{the right quasi-optimality rule (rQO)},
\end{alignat*}
where $\xadii$ is the second Bregman iterate, $D^\text{sym}$ and $D$ denote the symmetric and regular Bregman distances, all of which will be defined in the following.

Note that the heuristic discrepancy rule is sometimes also referred to as the Hanke-Raus rule (as the rules coincide for Landweber iteration). For clarity, it is preferable to name this method as the heuristic analogue of the classical discrepancy rule. In particular, this rule is the only one for which some convergence analysis has been done (cf.~\cite{jinlorenz,QJinLagrangian,QJinHankeRausBanach,QJinBregIt}). A discrete version of the quasi-optimality rules was also investigated in \cite{jinlorenz}. As mentioned, a numerical study of some rules was also done in \cite{kindermannmutimbu}.

These rules are well established in the linear case. However, except for the HD rule, their extension to the convex setting is certainly not obvious. We consider Bregman iteration as the natural analogue of Tikhonov iteration and we opt to define the latter three rules utilising the second Bregman iterate. The HR rule was considered in \cite{kindermannmutimbu}, whilst the quasi-optimality rules considered here are entirely novel.

Note that beside the right quasi-optimality rule it is possible to 
define a \textquotedblleft left\textquotedblright  quasi-optimality version using the functional
$D_{\xi^{II}_{\alpha,\delta}}(x^\delta_\alpha,x^{II}_{\alpha,\delta})$. 
However, the numerical results for this rule were quite subpar (as was demonstrated in \cite{kindermannmutimbu}), so we 
decided not to consider it further.

\subsection{Preliminaries}
Note that in contrast with linear regularisation theory, one cannot (in general) prove convergence of the regularised solution to $x^\dagger$ in the norm, but it is common to use the Bregman distance (cf.~\cite{bregman1967relaxation,osherburger}):
\[
D_\xi(x^\delta_\alpha,x^\dagger):=\mathcal{R}(x^\delta_\alpha)-\mathcal{R}(x^\dagger)-\langle \xi,x^\delta_\alpha-x^\dagger\rangle_{X^\ast\times X}.
\]
The Bregman distance is not a distance (a.k.a. a metric) as it does not satisfy the triangle inequality; nor is it in general symmetric. We do, however, have the following useful so-called three point identity (cf.~\cite{kindermannbanach}):
\begin{equation}
D_\xi(x^\delta_\alpha,x^\dagger)=D_{\xi_\alpha}(x^\delta_\alpha,x_\alpha)+D_\xi(x_\alpha,x^\dagger)+\langle\xi_\alpha-\xi,x^\delta_\alpha-x_\alpha\rangle.\label{threepoint}
\end{equation}
For any $\xi_1\in\partial\mathcal{R}(x_1)$ and $\xi_2\in\partial\mathcal{R}(x_2)$, one can also define the \emph{symmetric Bregman distance} as
    \begin{align*}
        D^{\text{sym}}_{\xi_1,\xi_2}:X\times X&\to\mathbb{R}\cup\{\infty\}
        \\
        (x_1,x_2)&\mapsto\langle\xi_1-\xi_2,x_1-x_2\rangle.
    \end{align*}
In the following, we use a super/sub-scripted $\delta$ to indicate variables associated with noisy data and its absence indicates the corresponding variables for exact data. For instance,
\[
x^\delta_\alpha\in\argmin_{x\in X}\mathcal{T}^\delta_\alpha(x)\qquad\text{and}\qquad x_\alpha\in\argmin_{x\in X}\mathcal{T}_\alpha(x),
\]   
where $\mathcal{T}_\alpha$ indicates the Tikhonov functional with exact data $y$ replacing $y^\delta$.
It is useful to define residuals as variables. In particular,
\[
\pad:=y^\delta-Ax^\delta_\alpha\qquad\text{and}\qquad\pa:=y-Ax_\alpha.
\]
The following  proposition can be proven via standard convex analysis:
\begin{proposition}
    The residual $\pad$ may be expressed in terms of a proximal mapping operator,
    \begin{equation}
  	\operatorname{prox}_{\mathcal{J}}:
  	Y\to Y, \qquad 
	\operatorname{prox}_{\mathcal{J}}=(I+\partial\mathcal{J})^{-1},\label{proximalsubdiff}
	\end{equation}
	in the form
	\begin{equation}\label{pr}
	p^\delta_\alpha:= \prox_{\mathcal{J}}(y^\delta),\qquad\text{with }\mathcal{J}=\alpha \mathcal{R}^\ast\circ \frac{1}{\alpha} A^\ast.
	\end{equation}
\end{proposition}
Note that in the above proposition, $A^\ast:Y\to X^\ast$ denotes the adjoint operator and $\mathcal{R}^\ast$ denotes the \emph{Fenchel conjugate} (cf.~\cite{beckfista,convexmonotone}).
We will make use of the firm non-expansivity of the proximal mapping operator:
\begin{equation}\label{firmnonexpansivity}
    \langle\prox_\mathcal{J}(y_1)-\prox_\mathcal{J}(y_2),y_1-y_2\rangle\ge\|\prox_\mathcal{J}(y_1)-\prox_{\mathcal{J}}(y_2)\|^2,
\end{equation}
cf.~\cite{beckfista,convexmonotone}.

The optimality condition for the Tikhonov functional may be stated as follows:
\begin{equation}
0\in A^\ast(Ax^\delta_\alpha-y^\delta)+\alpha\partial\mathcal{R}(x^\delta_\alpha), \label{tikhopt}
\end{equation}
for all $\alpha\in(0,\alpha_\text{max})$ and $y^\delta\in Y$. We may thus define 
\[
	\xi_\alpha:=\partial\mathcal{R}(x^\delta_\alpha)= -\frac{1}{\alpha}A^\ast(Ax_\alpha-y)\qquad\text{and}\qquad
	\xi^\delta_\alpha:=\partial\mathcal{R}(x_\alpha)=-\frac{1}{\alpha}A^\ast(Ax^\delta_\alpha-y^\delta),
\]
for the subgradients of $\mathcal{R}$ at $x_\alpha$ and $x^\delta_\alpha$, respectively.

The nonlinear analogue of iterated Tikhonov regularisation, as stated earlier, may be defined as \emph{Bregman iteration} (cf.~\cite{bregman1967relaxation,osher2005iterative}).  In particular, the \emph{second Bregman iterate} may be computed as
\begin{equation}
x^{II}_{\alpha,\delta}\in\argmin_{x\in X}\frac{1}{2}\|Ax-y^\delta\|^2+\alpha D_{\xi^\delta_\alpha}(x,x^\delta_\alpha).\label{bregmaniterate}
\end{equation}
As observed in \cite{yinbregman}, we can also compute $x^{II}_{\alpha,\delta}$ by minimising a simpler expression which does not involve the Bregman distance. Similarly as for the Tikhonov functional, we can state the optimality condition for the subsequent Bregman functional in the same manner:
    \[
    0\in A^\ast(Ax^{II}_{\alpha,\delta}-y^\delta-(y^\delta-Ax^\delta_\alpha))+\alpha\partial\mathcal{R}(x^{II}_{\alpha,\delta}),
    \]
and analogously, $x^{II}_\alpha$ for exact data. As before, we introduce the corresponding expressions for the subgradients
\[
    \xi^{II}_\alpha:=-\frac{1}{\alpha}A^\ast(Ax^{II}_{\alpha}-y-(y-Ax_\alpha))\qquad\text{and}\qquad
    \xi^{II}_{\alpha,\delta}:=-\frac{1}{\alpha}A^\ast(Ax^{II}_{\alpha,\delta}-y^\delta-(y^\delta-Ax^\delta_\alpha)),
\]
of $\mathcal{R}$ at $x^{II}_\alpha$ and $x^{II}_{\alpha,\delta}$, respectively.

The residual with respect to the second Bregman iterate may also be written in terms of the proximal point mapping:
\[
    p^{II}_{\alpha,\delta}:= y^\delta-Ax^{II}_{\alpha,\delta}=\prox_{\mathcal{J}}\left(y^\delta+p^\delta_\alpha\right)-p^\delta_\alpha\qquad\text{and}\qquad p^{II}_{\alpha}:= y-Ax^{II}_{\alpha}=\prox_{\mathcal{J}}\left(y+p_\alpha\right)-p_\alpha,
\]
for all $\alpha\in(0,\alpha_\text{max})$ and $y^\delta\in Y$ and $\mathcal{J}$ 
as in \eqref{pr}.  For notational purposes, we define the following
\begin{align*}
    \Dy = \yd - y,\qquad \Dp=\pad-\pa,\qquad \Dpii=\padii-p^{II}_\alpha.
\end{align*}
We state a useful estimate:
\begin{lemma} \label{Lemma1}
We have the following upper bound for the data propagation error:
\begin{equation} \label{funccal}
 D_{\xi_\alpha}(x^\delta_\alpha,x_\alpha) \leq 
\frac{1}{\alpha} \langle \Dy  - \Dp, \Dp\rangle 
\end{equation}
for all $\alpha\in(0,\alpha_{\text{max}})$ and $y,y^\delta\in Y$.
\end{lemma} 
\begin{proof}
We may estimate
\begin{equation}
\begin{aligned}
D_{\xi_\alpha}(x^\delta_\alpha,x_\alpha)&\le-\frac{1}{2\alpha}\|A(x^\delta_\alpha-x_\alpha)\|^2-\frac{1}{\alpha}\langle y-y^\delta,A(x^\delta_\alpha-x_\alpha)\rangle
\\
&\le -\frac{1}{\alpha}\langle A(x^\delta_\alpha-x_\alpha),A(x^\delta_\alpha-x_\alpha)\rangle-\frac{1}{\alpha}\langle y-y^\delta,A(x^\delta_\alpha-x_\alpha)\rangle
\\
&=-\frac{1}{\alpha}\langle Ax^\delta_\alpha-y^\delta+y-Ax_\alpha,A(x^\delta_\alpha-x_\alpha)\rangle, \label{usefulinequality}
\end{aligned}
\end{equation}
which proves the desired result.
\end{proof}
\subsection{Error estimates}
Convergence results for convex regularisation are well known. We state some standard results \cite{osherburger,variationalimaging}: we assume henceforth a regularity condition on $x^\dagger$; namely, that it fulfills the following \emph{source condition}:
\begin{equation}
\partial\mathcal{R}(x^\dagger)\in\operatorname{range}(A^\ast)\iff\exists\, w:A^\ast w\in\partial\mathcal{R}(x^\dagger). \label{source}
\end{equation}
Subsequently, $\xi:=A^\ast w$ is the subgradient of $\mathcal{R}$ at $x^\dagger$.
We remark that much of the analysis 
(concerning the convergence results, 
not the rates) 
below is valid with \eqref{source} replaced by weaker conditions, e.g., in form of 
variational source conditions 
\cite{HoKaPoSc,Fl,FlHo}.
We have the following error estimates:
\begin{proposition}
Let $x^\dagger$ satisfy the source condition \eqref{source}. Then
\begin{alignat*}{2}
D_\xi(x_\alpha,x^\dagger)&\le\frac{\|w\|^2}{2}\alpha,&\qquad
\|Ax_\alpha-y\|&\le 2\|w\|\alpha,
\\
D_{\xi_\alpha}(x^\delta_\alpha,x_\alpha)&\le\frac{\delta^2}{2\alpha},&\qquad
\|A(x^\delta_\alpha-x_\alpha)\|&\le 2\delta,
\end{alignat*}
and
\begin{align*}
D_\xi(x^\delta_\alpha,x^\dagger)\le\frac{1}{2}\left(\frac{\delta}{\sqrt{\alpha}}+\sqrt{\alpha}\|w\|\right)^2,\qquad
\|Ax^\delta_\alpha-y^\delta\|\le\delta+2\|w\|\alpha,
\end{align*}
for all $\alpha\in(0,\alpha_\text{max})$ and $y,y^\delta\in Y$.
\end{proposition}
Note also that
	\begin{equation}\label{PropResFuncEstimates}
	\|Ax^{II}_{\alpha,\delta}-y^\delta\|\le\|Ax^\delta_\alpha-y^\delta\|\qquad\text{and}\qquad\mathcal{R}(x^\delta_\alpha)\le\mathcal{R}(x^{II}_{\alpha,\delta}),
	\end{equation}
	for all $\alpha\in(0,\alpha_\text{max})$ and $y^\delta\in Y$. The inequalities 
	\eqref{PropResFuncEstimates}, of course, holds analogously for the noise-free variables. Note that from \eqref{source} and \eqref{threepoint}, we obtain the following estimate:
\begin{equation}
		D_\xi(x^\delta_{\alpha_\ast},x^\dagger)\le D_{\xi_{\alpha_\ast}}(x^\delta_{\alpha_\ast},x_{\alpha_\ast})+D_\xi(x_{\alpha_\ast},x^\dagger)+6\|w\|\delta,   \label{bregmantriangle}
		\end{equation} 
cf.~\cite{jinlorenz}. We will utilise this in the convergence proofs to come in the latter sections.

\subsection{The heuristic discrepancy rule}
In terms of the residual variables, the heuristic discrepancy functional may be expressed as
\begin{equation*}
\psi_{\text{HD}}(\alpha,y^\delta)=\frac{\|\pad\|^2}{\alpha},
\end{equation*}
and $\alpha_\ast$ is selected as its minimiser.
In the paper \cite{jinlorenz}, the following error estimate was derived:
\begin{theorem}
	Let $x^\dagger$ satisfy the source condition \eqref{source}, let $\alpha_\ast$ be chosen according to the heuristic discrepancy rule and suppose that $\delta^\ast:=\|Ax^\delta_\alpha-y^\delta\|\ne 0$. Then there exists a constant $C>0$ such that
	\begin{equation*}
	D_\xi(x^\delta_{\alpha_\ast},x^\dagger)\le C\left(1+\left(\frac{\delta}{\delta^\ast}\right)^2\right)\max\{\delta,\delta^\ast\}.
	\end{equation*}
\end{theorem}
The above estimate is of restricted utility as one has no control over the value of $\delta^\ast$. If one assumes the condition
\begin{equation}
\|Q(y-y^\delta)\|\ge\varepsilon\|y-y^\delta\|,\label{jinnoisecondition}
\end{equation}
where $Q:Y\to(\overline{\text{range}(A)})^\perp$ is an orthogonal projection and $\varepsilon>0$, which was introduced by Hanke and Raus in \cite{hankeraus}, then one can prove that
\begin{equation}
D_\xi(x^\delta_{\alpha_\ast},x^\dagger)\le C\left(1+\frac{1}{\varepsilon^2}\right)\max\{\delta,\delta^\ast\};\label{jinbregmanestimate}
\end{equation}
thus without the troublesome prefactor  and using \eqref{jinnoisecondition} and \eqref{jinbregmanestimate}, it is possible to prove convergence of the regularisation method, i.e.,
\begin{equation}
D_\xi(x^\delta_{\alpha_\ast},x^\dagger)\to 0\qquad \text{as } \delta\to 0.\label{hankerausconvergence}
\end{equation}
The condition \eqref{jinnoisecondition} is quite abstract and if one considers the case in which $\overline{\text{range}(A)}=Y$, it would follow that $Q:Y\to\emptyset$ and consequently $\|Q(y-y^\delta)\|=0\ge\varepsilon\|y-y^\delta\|$, i.e., the condition is not satisfied.

Another noise condition was subsequently postulated in \cite{jinlorenz}; namely, if there exists $\varepsilon\in(0,1)$ such that $y^\delta-y\in\mathcal{N}$, where
\begin{equation}
\mathcal{N}:=\bigg\{\epsilon \in Y:\langle \epsilon,z\rangle\le(1-\varepsilon)\|\epsilon\|\|z\|\text{ for all }z\in\overline{A(\operatorname{dom}\partial\mathcal{R})} \bigg\};\label{jinlorenznewnoise}
\end{equation}
then, if $y^\delta-y\in\mathcal{N}$, it would follow that
	\begin{equation}
	D_\xi(x^\delta_{\alpha_\ast},x^\dagger)\le C\left(1+\frac{1}{1-(1-\varepsilon)^2}\right)\max\{\delta,\delta^\ast\}, \label{newjinestimate}
	\end{equation}
from which it is again possible to prove convergence. It still remains unclear when \eqref{jinlorenznewnoise} is satisfied.
It is actually the main aim of this 
paper to replace these conditions with 
more practical ones. 

%We opt to define a different noise condition from which one can offer theoretical explanations for when it is satisfied; namely, there exists a positive constant $C>0$ such that
%\begin{equation}
%D_{\xi_\alpha}(x^\delta_\alpha,x_\alpha)\le C\frac{\|\Dp\|^2}{\alpha},\label{newnoisecondition}
%\end{equation}
%and with this condition, prove \eqref{hankerausconvergence}. One may observe that it resembles the noise condition of \cite{kinderabstract} in that the data propogation is bounded from above by a functional with the difference of residuals. Certainly, in this form, it is still an \emph{abstract} condition, although in we will reduce it in the following to a more reasonable form from which we can extract more understanding. Note that we shall refer to conditions of the type \eqref{newnoisecondition} as the \emph{auto-regularisation} condition.
For heuristic rules, it is often standard to show convergence of the selected parameter as the noise level tends to zero:
\begin{proposition}
	Let $\alpha_\ast$ be the minimiser of the heuristic discrepancy functional. Then $\alpha_\ast\to 0$ as $\delta\to 0$.
\end{proposition}
For the proof, we refer to \cite{jinlorenz}.
In order to prove convergence, the most difficult part is to derive a condition that prevents $\alpha_\ast$ from decaying too rapidly. This always involves some restriction on the noise. In the next theorem, we impose such a noise restriction in the form of an \emph{auto-regularisation} condition:

\begin{theorem}\label{th:two}
	Let the source condition \eqref{source} be satisfied, $\alpha_\ast$ be selected according to the heuristic discrepancy rule, $\delta^\ast\ne 0$ and assume the auto-regularisation condition 
	\begin{equation}
D_{\xi_\alpha}(x^\delta_\alpha,x_\alpha)\le C\frac{\|\Dp\|^2}{\alpha},\label{newnoisecondition}
\end{equation}
	holds for all $\alpha\in(0,\alpha_\text{max})$ and $y^\delta\in Y$. Then it follows that the method converges; i.e.,
	\[
	D_\xi(x^\delta_{\alpha_\ast},x^\dagger)\to 0,
	\]
	as $\delta\to 0$.
\end{theorem}
\begin{proof}
	From the estimate \eqref{bregmantriangle}, one may then immediately estimate the data propagation error courtesy of the auto-regularisation condition as
		\begin{equation*}
		D_{\xi_{\alpha_\ast}}(x^\delta_{\alpha_\ast},x_{\alpha_\ast})\le C\frac{\|\Dp_\ast\|^2}{\alpha_\ast}.
		\end{equation*}
		Since $\alpha_\ast$ minimises the heuristic discrepancy functional, it follows that
		\begin{equation*}
		\begin{aligned} 
		\psi_{\text{HD}}(\alpha_\ast,y^\delta)&\le\psi_{\text{HD}}(\alpha,y^\delta)=\frac{\|p^\delta_\alpha\|^2}{\alpha}\le\frac{(\delta+2\alpha\|w\|)^2}{\alpha}=\frac{\delta^2+4\alpha\|w\|\delta+4\alpha^2\|w\|^2}{\alpha}
		\\
		&=\frac{\delta^2}{\alpha}+4\|w\|\delta+4\alpha\|w\|^2=\left(\frac{\delta}{\sqrt{\alpha}}+2\|w\|\sqrt{\alpha}\right)^2.
		\end{aligned}
		\end{equation*}
		Hence, one may choose $\alpha=\alpha(\delta)$ such that $\alpha(\delta)\to 0$ as $\delta^2/\alpha(\delta)\to 0$ for $\delta\to 0$. Furthermore,
		\begin{equation*}
		\begin{aligned}
		\psi_{\text{HD}}(\alpha_\ast,y)&=\frac{\|p_{\alpha_\ast}\|^2}{\alpha_\ast}\le \frac{4\|w\|^2\alpha^2_\ast}{\alpha_\ast}=4\|w\|^2\alpha_\ast\xrightarrow{\delta\to 0}0.
		\end{aligned}
		\end{equation*}
		Thus, we may conclude that
		\begin{equation*}
		\begin{aligned}
		\frac{\|\Dp_\ast\|^2}{\alpha_\ast}&\le\left(\frac{\|p^\delta_{\alpha_\ast}\|}{\sqrt{\alpha_\ast}}+\frac{\|p_{\alpha_\ast}\|}{\sqrt{\alpha_\ast}}\right)^2=\left(\sqrt{\psi_{\text{HD}}(\alpha_\ast,y^\delta)}+\sqrt{\psi_{\text{HD}}(\alpha_\ast,y)}\right)^2\xrightarrow{\delta\downarrow 0}\left(\sqrt{0}+\sqrt{0}\right)^2=0.
		\end{aligned}
		\end{equation*}
		For the approximation error, it follows that
		\begin{equation*}
		D_\xi(x_{\alpha_\ast},x^\dagger)\le\frac{\|w\|^2}{2}\alpha_\ast\xrightarrow{\delta\to 0}0.
		\end{equation*}
		Hence each term in \eqref{bregmantriangle} tends to 0 as $\delta\to 0$.
\end{proof}
From Lemma~\ref{Lemma1}, it suffices for \eqref{newnoisecondition} to show that there exists $C>0$ such that
\begin{equation}
\langle \Dp, \Dy -\Dp\rangle \leq C \|\Dp\|^2
%\langle Ax^\delta_\alpha-y^\delta-(Ax_\alpha-y),A(x_\alpha-x^\delta_\alpha)\rangle\le C\|Ax^\delta_\alpha-y^\delta-(Ax_\alpha-y)\|^2. 
\end{equation}
%Note that the LHS of the above inequality is
%\begin{align*}
%    \langle Ax^\delta_\alpha-y^\delta-(Ax_\alpha-y),Ax_\alpha-Ax^\delta_\alpha\rangle&=\langle Ax^\delta_\alpha-y^\delta-(Ax_\alpha-y),Ax_\alpha-y-(Ax^\delta_\alpha-y^\delta)+y-y^\delta\rangle
%    \\
%    &=\langle Ax^\delta_\alpha-y^\delta-(Ax_\alpha-y),y-y^\delta\rangle-\|Ax^\delta_\alpha-y^\delta-(Ax_\alpha-y)\|^2.
%\end{align*}
Obviously, it is then enough to prove
\begin{equation}
\langle \Dp,\Dy\rangle\le C\|\Dp\|^2.
\label{hdabstractineq}
\end{equation} for some positive constant $C$.

The auto-regularisation condition is an implicit condition on the noise. One may observe that it resembles the condition of \cite{kinderabstract} in the linear case. Certainly, in this form, it is still an \emph{abstract} condition, although we will reduce it in Section~\ref{DiagSec} to a more reasonable form from which we can extract more understanding. 

\subsubsection{Convergence rates}
With the aid of the source condition, auto-regularisation condition and an additional regularity condition, we can even derive rates of convergence.
We start with the following proposition:
\begin{proposition}\label{propraz}
Suppose that $\partial \mathcal{R}^\ast: X^\ast\to X^{\ast\ast}$
is continuous at $0$.
%
%for any $z \in X'$ we have 
%\begin{equation}\label{ccon}
% \limsup_{\|z\|_{X'}\to 0} \frac{\|\partial R^*(z)\|_{X''}}{\|z\|_{X'}} < \infty.      
%\end{equation}
Then, for all $\yd$ with $\|\yd\|\geq C$, there is a constant such that 
\begin{equation}\label{HDConRat}
\alpha \leq C \frac{1}{\alpha}\|\pad\|^2   \qquad \forall \alpha \in (0,\alpha_{\text{max}}). \end{equation}
\end{proposition}
\begin{proof}
As $\langle\pad,\yd\rangle\geq 0$, we have
\begin{align*} \|\yd\|^2  &= \langle\pad,\yd\rangle + \langle
\partial \mathcal{R}^*(A^\ast\frac{\pad}{\alpha}),A^\ast y\rangle  \leq 
\alpha_\text{max} \langle\frac{\pad}{\alpha},\yd\rangle + \|\partial \mathcal{R}^*(A^\ast\frac{\pad}{\alpha})\|_{X^{\ast\ast}}
\|A^\ast\|_{Y,X^\ast}\|\yd\| \\
& \leq \|\yd\| \left( 
\alpha_\text{max}  \left\|\frac{\pad}{\alpha}\right\|  + \|\partial \mathcal{R}^*(A^\ast\frac{\pad}{\alpha})\|_{X^{\ast\ast}}
%{\|A^\ast\frac{\pad}{\alpha}\|_{X'}}
\|A^\ast\|_{Y,X^\ast}%\|\frac{\pad}{\alpha})\| 
\right).%\\ &= 
%\|\yd\| \|\frac{\pad}{\alpha})\|\left(  \alpha_0 % 
%C + \frac{\|\partial R^*(A^\ast\frac{\pad}{\alpha})\|_{X''}}{\|A^\ast\frac{\pad}{\alpha}\|_{X'}}\right).
\end{align*}
Now assume that \eqref{HDConRat} does not hold. Then there is a sequence 
with $\|\yd\|\geq C$ and $\|\frac{\pad}{\alpha}\|\to 0$. 
However, this implies that $A^\ast\frac{\pad}{\alpha}\to 0$ in $X^\ast$ and 
by  continuity, 
%\eqref{ccon}
also that 
\[
\|\partial \mathcal{R}^*(A^\ast\frac{\pad}{\alpha})\|_{X^{\ast\ast}}\to 0. \] 
However, this leads to a contradiction as then  
\[\|\yd\|^2 \leq C \|\yd\| \|\frac{\pad}{\alpha}\| \to 0. \]
Thus \eqref{HDConRat} must hold. 
\end{proof}
We now state the main convergence rates result for the HD rule:
\begin{proposition}
	Let the source condition \eqref{source} hold, $\alpha_\ast$ be 
	selected according to the heuristic discrepancy rule and suppose the auto-regularisation condition \eqref{newnoisecondition} is satisfied. Assume, in addition, that $\|\yd\| \geq  C$
	and that $\partial \mathcal{R}^*:X^\ast\to X^{\ast\ast}$ is continuous at $0$. 
	Then
	\[
	D_\xi(x^\delta_{\alpha_\ast},x^\dagger)=\mathcal{O}(\delta),
	\]
	for $\delta>0$ sufficiently small.
\end{proposition}
\begin{proof}
%	First, we seek to verify that there exists a positive constant such that $\alpha_\ast\le C\psi_{\text{HD}}(\alpha_\ast,y^\delta)$. To this end, observe that for all $\alpha>0$, one has
%\begin{align*}
%&\alpha \frac{p^\delta_\alpha}{\alpha}+\alpha A\partial\mathcal{R}^\ast(A^\ast\frac{p^\delta_\alpha}{\alpha})=y^\delta
%\\
%\iff& q^\delta_\alpha+A\partial\mathcal{R}^\ast(A^\ast q^\delta_\alpha)=\frac{y^\delta}{\alpha},\qquad\text{if }q^\delta_\alpha:=\frac{p^\delta_\alpha}{\alpha},
%\\
%\iff&q^\delta_\alpha+\partial(\mathcal{R}^\ast\circ %A^\ast)(q^\delta_\alpha)=\frac{y^\delta}{\alpha}.
%\end{align*}
%Thus, it suffices to show that there exists a positive constant $C$ such that
%\[
%\|q^\delta_{\alpha_\ast}\|^2\ge C;
%\]
%i.e., that $q^\delta_{\alpha_\ast}\ne 0\iff p^\delta_{\alpha_\ast}\ne 0$.
%	
Note that from Proposition~\ref{propraz}
and since $\alpha_\ast$ is the global minimiser, for any $\alpha$, we have that 
$\alpha_\ast\le C\psi_{\text{HD}}(\alpha_\ast,y^\delta)\leq 
C \psi_{\text{HD}}(\alpha,y^\delta)$.
Observe that from \eqref{bregmantriangle}, it follows that
\begin{align*}
    D_{\xi}(x^\delta_{\alpha_\ast},x^\dagger)&\le D_{\xi_{\alpha_\ast}}(x^\delta_{\alpha_\ast},x_{\alpha_\ast})+\frac{\|w\|^2}{2}\alpha_\ast+6\|w\|\delta
    \\
    &\le\left(\sqrt{\psi_{\text{HD}}(\alpha,y^\delta)}+\sqrt{\psi_{\text{HD}}(\alpha_\ast,y)} \right)^2+C\delta+C\alpha_\ast
    \\
    &\le\left(\frac{\delta}{\sqrt{\alpha}}+C\sqrt{\alpha}+C\sqrt{\alpha_\ast} \right)^2+C\delta+C\alpha_\ast
    \\
    &=\mathcal{O}\left(\left(\frac{\delta}{\sqrt{\alpha}}+\sqrt{\alpha}\right)^2+\frac{\delta^2}{\alpha}+\alpha+\delta \right)\qquad\text{since }\alpha_\ast\le C\psi_{\text{HD}}(\alpha,y^\delta),
    \\
    &=\mathcal{O}\left(\delta\right),
\end{align*}
choosing $\alpha=\alpha(\delta)=\delta$.
\end{proof}

%\paragraph{Noise Condition}

\subsection{The Hanke-Raus rule}	
As with the HD rule, the Hanke-Raus functional may be reexpressed as
\begin{equation*}
\psi_{\text{HR}}(\alpha,y^\delta)=\frac{1}{\alpha}\langle \padii,\pad \rangle.
\end{equation*}
To the best of the authors' knowledge, in contrast to the heuristic discrepancy rule, the Hanke-Raus rule has not yet been rigorously analysed in the convex variational setting although it has been tested numerically for total variation regularisation (cf.~\cite{kindermannmutimbu}).

Note that the Hanke-Raus functional is not expressed in terms of a norm, thus there is no a-priori guarantee that it remains positive (as it is in the linear case). We therefore provide the following proposition:
\begin{proposition}
	We have that
	\[
	\psi_{\text{HR}}(\alpha,y^\delta)\ge 0,
	\]
	for all $\alpha\in(0,\alpha_\text{max})$ and $y^\delta\in Y$.
\end{proposition}
\begin{proof}
	We obtain that
	\begin{align*}
	\psi_{\text{HR}}(\alpha,y^\delta)&
	%=\frac{1}{\alpha}\langle p^{II}_{\alpha,\delta},p^\delta_\alpha\rangle
	=\frac{1}{\alpha}\left\langle\prox_{\mathcal{J}}(y^\delta+p^\delta_\alpha)-p^\delta_\alpha,p^\delta_\alpha\right\rangle
	=\frac{1}{\alpha}\left\langle\prox_{\mathcal{J}}(y^\delta+p^\delta_\alpha)-\prox_{\mathcal{J}}(y^\delta),p^\delta_\alpha\right\rangle
	\\
	&=\frac{1}{\alpha}\left\langle\prox_{\mathcal{J}}(y^\delta+p^\delta_\alpha)-\prox_{\mathcal{J}}(y^\delta),(y^\delta+p^\delta_\alpha)-y^\delta\right\rangle\ge 0,
	\end{align*}
	for all $\alpha\in(0,\alpha_\text{max})$, which follows from \eqref{firmnonexpansivity}.
\end{proof}

\begin{proposition}
	We have that
	\[
	\psi_{\text{HR}}(\alpha,y^\delta)\le\psi_{\text{HD}}(\alpha,y^\delta),
	\]
	for all $\alpha\in(0,\alpha_\text{max})$ and $y^\delta\in Y$.
\end{proposition}
\begin{proof}
	We can estimate
	\begin{align*}
	\psi_{\text{HR}}(\alpha,y^\delta)&=\frac{1}{\alpha}\left\langle\prox_{\mathcal{J}^\ast}(y^\delta+p^\delta_\alpha)-\prox_{\mathcal{J}^\ast}(y^\delta),y^\delta+p^\delta_\alpha-y^\delta\right\rangle
	\\
	&\le\frac{1}{\alpha}\|\prox_{\mathcal{J}^\ast}(y^\delta+p^\delta_\alpha)-\prox_{\mathcal{J}^\ast}(y^\delta)\|\|p^\delta_\alpha\|
	\le\frac{1}{\alpha}\|p^\delta_\alpha\|^2
	=\psi_{\text{HD}}(\alpha,y^\delta),
	\end{align*}
	for all $\alpha\in(0,\alpha_\text{max})$, where we have used \eqref{firmnonexpansivity}.
\end{proof}

\begin{proposition}
	Let $\alpha_\ast$ be the minimiser of the Hanke-Raus functional. Then $\alpha_\ast\to 0$ as $\delta\to 0$.
\end{proposition}
\begin{proof}
	Since $\psi_{\text{HR}}(\alpha_\ast,y^\delta)\le\psi_{\text{HD}}(\alpha_\ast,y^\delta)\to 0$ as $\delta\to 0$, we deduce that $\|Ax^\delta_{\alpha_\ast}-y^\delta\|\to 0$ as $\delta\to 0$. Then the proof is the same as the one given in \cite{jinlorenz} for $\alpha_\ast$ selected according to the heuristic discrepancy rule.
\end{proof}
Next we state the main convergence theorem for the Hanke-Raus rule for which we again require an auto-regularisation condition:
\begin{theorem}\label{th:three}
	Let the source condition \eqref{source} be satisfied and let $\alpha_\ast$ be the minimiser of $\psi_{HR}(\alpha,y^\delta)$ and suppose that there exists a positive constant $C>0$ such that
	\begin{equation}
	\begin{aligned}
	D_{\xi_\alpha}(x^\delta_\alpha,x_\alpha)&\le C\frac{1}{\alpha}\langle \Dpii,\Dp\rangle
	,
	\end{aligned}  \label{noiseconditionhankeraus}
	\end{equation}
	for all $\alpha\in(0,\alpha_\text{max})$ and $y^\delta\in Y$. Then
	\begin{equation*}
	D_\xi(x^{\delta}_{\alpha_\ast},x^\dagger)\to 0\text{ as }\delta\to 0.
	\end{equation*}
\end{theorem}
\begin{proof}
	Let $\alpha_\ast$ be the minimiser of the Hanke-Raus functional. Then, as before, we estimate the Bregman distance as \eqref{bregmantriangle} and from \eqref{noiseconditionhankeraus}, deduce that
	\[
	\begin{aligned}
	D_{\xi_{\alpha_\ast}}(x^\delta_{\alpha_\ast},x_{\alpha_\ast})&\le\frac{1}{\alpha_\ast}\langle \Dpiiast,\Dp_\ast\rangle
	=\frac{1}{\alpha_\ast}\langle p^{II}_{\alpha_\ast},\Dp_\ast\rangle-\frac{1}{\alpha_\ast}\langle p^{II}_{\alpha_\ast,\delta},\Dp_\ast\rangle
	\\
	&=\frac{1}{\alpha_\ast}\langle p^{II}_{\alpha_\ast},p_{\alpha_\ast}\rangle-\frac{1}{\alpha_\ast}\langle p^{II}_{\alpha_\ast},p^\delta_{\alpha_\ast}\rangle+\frac{1}{\alpha_\ast}\langle p^{II}_{\alpha_\ast,\delta},p^\delta_{\alpha_\ast}\rangle-\frac{1}{\alpha_\ast}\langle p^{II}_{\alpha_\ast,\delta},p_{\alpha_\ast}\rangle
	\\
	&\le \psi_{\text{HR}}(\alpha,y^\delta)+\psi_{\text{HR}}(\alpha_\ast,y)-\frac{1}{\alpha_\ast}\langle p^{II}_{\alpha_\ast},p^\delta_{\alpha_\ast}\rangle-\frac{1}{\alpha_\ast}\langle p^{II}_{\alpha_\ast,\delta},p_{\alpha_\ast}\rangle.
	\end{aligned}
	\]
	Now, notice that the last two terms can be estimated as
	\[
	\begin{aligned}
	&-\frac{1}{\alpha_\ast}\langle p^{II}_{\alpha_\ast},p^\delta_{\alpha_\ast}\rangle-\frac{1}{\alpha_\ast}\langle p^{II}_{\alpha_\ast,\delta},p_{\alpha_\ast}\rangle\le\frac{1}{\alpha_\ast}\|p^{II}_{\alpha_\ast}\|\|p^\delta_{\alpha_\ast}\|+\frac{1}{\alpha_\ast}\|p^{II}_{\alpha_\ast,\delta}\|\|p_{\alpha_\ast}\|\le\frac{2}{\alpha_\ast}\|p_{\alpha_\ast}\|\|p^\delta_{\alpha_\ast}\|
	\\
	&\qquad\le 4\|w\|\delta+4\|w\|^2\alpha_\ast\to 0\text{ as }\delta\to 0.
	\end{aligned}
	\]
Moreover,
\[
\psi_{\text{HR}}(\alpha_\ast,y)\le\frac{1}{\alpha_\ast}\|p_{\alpha_\ast}\|^2\le 4\|w\|^2\alpha_\ast\to 0
\]
as $\delta\to 0$, and obviously
\[
\psi_{\text{HR}}(\alpha,y^\delta)\le\frac{\|p^\delta_\alpha\|}{\alpha}\to 0
\]
choosing $\alpha$ such that $\alpha\to 0$ and $\delta^2/\alpha\to 0$ as $\delta\to 0$. The result then follows from the fact that the other terms in \eqref{bregmantriangle} also vanish.
\end{proof}
By Lemma~\ref{Lemma1}, sufficient for  the auto-regularisation 
condition \eqref{funccal} is that 
\begin{align*}
&\langle \Dp,\Dy\rangle-\| \Dp\|^2\le C\langle \Dpii,\Dp\rangle.
\end{align*}
\subsubsection{Convergence rates}
The following  convergence rates theorem for the Hanke-Raus rule 
requires an additional condition, which we, however, will  not investigate 
further as it is beyond the scope of our work. 
\begin{proposition}
	Let the source condition \eqref{source} hold, $\alpha_\ast$ be selected according to the Hanke-Raus rule and supposes the auto-regularisation condition \eqref{noiseconditionhankeraus} is satisfied. Assume, in addition that
	$\|\yd\|\geq C$ for all $\delta$ and suppose that a constant exists with  
	\[ C \leq \left\langle \frac{p^{II}_{\alpha,\delta}}{\alpha},\frac{p^\delta_{\alpha}}{\alpha} \right\rangle, \qquad \forall \alpha \in (0,\alpha_\text{max})\quad\text{and}\quad \|\yd\| \geq C. \]
 Then
	\[
	D_\xi(x^\delta_{\alpha_\ast},x^\dagger)=\mathcal{O}(\delta),
	\]
	for $\delta>0$ sufficiently small.
\end{proposition}
\begin{proof}
	Notice that the imposed conditions imply that $\alpha_\ast \leq 
	 C\psi_{\text{HR}}(\alpha_\ast,y^\delta)$. 
	%\[
	%\alpha_\ast\le C\frac{1}{\alpha_\ast}\langle %p^{II}_{\alpha_\ast,\delta},p^\delta_{\alpha_\ast}\rangle,
	%\]
	%is equivalent to
	%\[
	%\langle q^{II}_{\alpha_\ast,\delta},q^\delta_{\alpha_\ast}\rangle\ge %C,\qquad\text{where }q^{II}_{\alpha,\delta}=\frac{p^{II}_{\alpha,\delta}}{\alpha},
	%\]
	%i.e. $\sgn(q^{II}_{\alpha_\ast,\delta})=\sgn(q^\delta_{\alpha_\ast})$ and %$0\notin\{q^{II}_{{\alpha_\ast},\delta},q^\delta_{\alpha_\ast} \}$.
	%
	Now,
\begin{align*}
	D_\xi(x^\delta_{\alpha_\ast},x^\dagger)&\le C\frac{1}{\alpha_\ast}\langle \Dpiiast,\Dp_\ast\rangle+C\alpha_\ast+C\delta
	=\mathcal{O}\left(\left(\frac{\delta}{\sqrt{\alpha}}+\sqrt{\alpha} \right)^2+\alpha_\ast+\delta \right)
	\\
	&=\mathcal{O}\left(\left(\frac{\delta}{\sqrt{\alpha}}+\sqrt{\alpha}\right)^2+\delta \right)\qquad\text{since }\alpha_\ast\le C\psi_{\text{HR}}(\alpha_\ast,y^\delta)
	\\
	&=\mathcal{O}\left(\delta^2+\delta\right)\qquad\text{ choosing }\alpha=\delta
	\\
	&=\mathcal{O}\left(\delta\right),
\end{align*}
for $\delta$ sufficiently small.
\end{proof}

%\paragraph{Noise condition}

\subsection{The quasi-optimality rules}
The principle behind the quasi-optimality rule is to minimise the difference of two successive approximations of the solution. In the linear case, the difference is measured with the norm, but in the convex setting, we use the Bregman distance. Therefore, possibilities for the quasi-optimality rule include choosing $\alpha_\ast$ as the minimiser of $D^{\text{sym}}_{\xi^{II}_{\alpha,\delta},\xi^\delta_\alpha}(x^{II}_{\alpha,\delta},x^\delta_\alpha)$, $D_{\xi^{II}_{\alpha,\delta}}(x^\delta_\alpha,x^{II}_{\alpha,\delta})$ or $D_{\xi^\delta_\alpha}(x^{II}_{\alpha,\delta},x^\delta_\alpha)$.

Similarly as for the Hanke-Raus rule, the authors are not aware of any other analysis of the quasi-optimality rule defined as above. The discrete version was considered in \cite{jinlorenz}. There, the noise condition postulated was a generalisation into the convex setting of the auto-regularisation set of \cite{glasko}, and the numerical performance of the rule with $D_{\xi^{II}_{\alpha,\delta}}(x^\delta_\alpha,x^{II}_{\alpha,\delta})$ was tested in \cite{kindermannmutimbu}. The performance of the latter rule in the aforementioned reference and also in our own numerical experiments proved to be quite poor and therefore we omit it. The rule with the symmetric Bregman distance performs reasonably, on the other hand. The remaining version is generally the best performing of all the quasi-optimality rules. More light will be shed on this, however, in the numerics section.

\subsubsection{The symmetric quasi-optimality rule}
Similarly to the previous two rules discussed, the symmetric quasi-optimality functional may also be expressed in terms of residuals:
\begin{proposition}
	We have that
	\begin{equation*}
	\psi_{SQO}(\alpha,y^\delta)= \frac{1}{\alpha}\langle 
	\pad -\padii,\padii \rangle,
	\end{equation*}
	for all $\alpha\in(0,\alpha_\text{max})$ and $y^\delta\in Y$.
\end{proposition}
\begin{proof}
	We have
	\begin{equation}
	\begin{aligned}
	&D^{\text{sym}}_{\xi^{II}_{\alpha,\delta},\xi^\delta_\alpha}(x^{II}_{\alpha,\delta},x^\delta_\alpha)=\langle\xi^{II}_{\alpha,\delta}-\xi^\delta_\alpha,x^{II}_{\alpha,\delta}-x^\delta_\alpha\rangle
	\\
	&\qquad =\frac{1}{\alpha}\langle A^\ast(Ax^\delta_\alpha-y^\delta)-A^\ast(Ax^{II}_{\alpha,\delta}-y^\delta+Ax^\delta_\alpha-y^\delta),x^{II}_{\alpha,\delta}-x^\delta_\alpha\rangle
	\\
	&\qquad =\frac{1}{\alpha}\langle A(x^\delta_\alpha-x^{II}_{\alpha,\delta}),Ax^{II}_{\alpha,\delta}-y^\delta\rangle
	=\frac{1}{\alpha}\langle Ax^\delta_\alpha-y^\delta-(Ax^{II}_{\alpha,\delta}-y^\delta),Ax^{II}_{\alpha,\delta}-y^\delta\rangle,
	\label{iteratedquasiopt}
	\end{aligned}
	\end{equation}
	for all $\alpha\in(0,\alpha_\text{max})$ and $y^\delta\in Y$, which is what we wanted to show.
\end{proof}

\begin{proposition}
	We have
	\[
	\psi_{\text{SQO}}(\alpha,y^\delta)\le\psi_{\text{HR}}(\alpha,y^\delta),
	\]
	for all $\alpha\in(0,\alpha_\text{max})$ and $y^\delta\in Y$.
\end{proposition}
\begin{proof}
	It follows trivially from the observation that
	\[
	\psi_{\text{SQO}}(\alpha,y^\delta)=\frac{1}{\alpha}\langle p^\delta_\alpha-p^{II}_{\alpha,\delta},p^{II}_{\alpha,\delta}\rangle=\frac{1}{\alpha}\langle p^\delta_\alpha,p^{II}_{\alpha,\delta}\rangle-\frac{1}{\alpha}\|p^{II}_{\alpha,\delta}\|^2\le\psi_{\text{HR}}(\alpha,y^\delta),
	\]
	for all $\alpha\in(0,\alpha_\text{max})$ and $y^\delta\in Y$.
\end{proof}

\begin{proposition}
	Let $\alpha_\ast$ be the minimiser of the symmetric quasi-optimality functional.
	Then $\alpha_\ast\to 0$ as $\delta\to 0$.
\end{proposition}
\begin{proof}
	Since $\psi_{\text{SQO}}(\alpha_\ast,y^\delta)\le\psi_{\text{HR}}(\alpha_\ast,y^\delta)\le\psi_{\text{HD}}(\alpha_\ast,y^\delta)$, the proof that $\alpha_\ast\to 0$ as $\delta\to 0$ is identical to the one given in \cite{jinlorenz} for the heuristic discrepancy rule.
\end{proof}
Similar as above, an auto-regularisation condition leads to convergence:
\begin{theorem}\label{th:four}
	Let the source condition \eqref{source} be satisfied and, $\alpha_\ast$ be the minimiser of $\psi_{\text{SQO}}(\alpha,y^\delta)$ and suppose that
	\begin{equation}
	D_{\xi_\alpha}(x^\delta_\alpha,x_\alpha)\le\frac{1}{\alpha}\langle \Dp-\Dpii,\Dpii \rangle,
	\label{noiseconditionquasiopt}
	\end{equation}
	for all $\alpha\in(0,\alpha_\text{max})$ and $y,y^\delta\in Y$.
	
	Then
	\begin{equation*}
	D_\xi(x^{\delta}_{\alpha_\ast},x^\dagger)\to 0\text{ as }\delta\to 0.
	\end{equation*}
\end{theorem}
\begin{proof}
	From \eqref{bregmantriangle} and \eqref{noiseconditionquasiopt},  it remains to estimate
	\[
	\begin{aligned}
	&D_{\xi_{\alpha_\ast}}(x^\delta_{\alpha_\ast},x_{\alpha_\ast})\\
	&\quad \le\frac{1}{\alpha_\ast}\langle p^\delta_{\alpha_\ast}-p^{II}_{\alpha_\ast,\delta},p^{II}_{\alpha_\ast,\delta}\rangle+\frac{1}{\alpha_\ast}\langle p_{\alpha_\ast}-p^{II}_{\alpha_\ast},p^{II}_{\alpha_\ast}\rangle-\frac{1}{\alpha_\ast}\langle p_{\alpha_\ast}-p^{II}_{\alpha_\ast},p^{II}_{\alpha_\ast,\delta} \rangle-\frac{1}{\alpha_\ast}\langle p^\delta_{\alpha_\ast}-p^{II}_{\alpha_\ast,\delta},p^{II}_{\alpha_\ast}\rangle
	\\
	&\quad \le\psi_{\text{SQO}}(\alpha,y^\delta)+\psi_{\text{SQO}}(\alpha_\ast,y)-\frac{1}{\alpha_\ast}\langle p_{\alpha_\ast}-p^{II}_{\alpha_\ast},p^{II}_{\alpha_\ast,\delta} \rangle-\frac{1}{\alpha_\ast}\langle p^\delta_{\alpha_\ast}-p^{II}_{\alpha_\ast,\delta},p^{II}_{\alpha_\ast}\rangle.
	\end{aligned}
	\]
	Now, the last two \textquotedblleft remainder\textquotedblright terms can be estimated from above by
	\[
	\begin{aligned}
	&\frac{1}{\alpha_\ast}\langle p^\delta_{\alpha_\ast},p^{II}_{\alpha_\ast} \rangle+\frac{1}{\alpha_\ast}\langle p^{II}_{\alpha_\ast,\delta},p^{II}_{\alpha_\ast} \rangle-\frac{1}{\alpha_\ast}\langle p_{\alpha_\ast},p^{II}_{\alpha_\ast,\delta}\rangle +\frac{1}{\alpha_\ast}\langle p^{II}_{\alpha_\ast}, p^{II}_{\alpha_\ast,\delta} \rangle
	\\
	&\le\frac{1}{\alpha_\ast}\|p^\delta_{\alpha_\ast}\|\|p^{II}_{\alpha_\ast}\|+\frac{1}{\alpha_\ast}\|p^{II}_{\alpha_\ast,\delta}\|\|p^{II}_{\alpha_\ast}\|+\frac{1}{\alpha_\ast}\|p_{\alpha_\ast}\|\|p^{II}_{\alpha_\ast,\delta}\|+\frac{1}{\alpha_\ast}\|p^{II}_{\alpha_\ast}\|\|p^{II}_{\alpha_\ast,\delta}\|
	\\
	&\le \frac{4}{\alpha_\ast}\|p^\delta_{\alpha_\ast}\|\|p_{\alpha_\ast}\|\le 8\|w\|\delta+8\|w\|^2\alpha_\ast\xrightarrow{\delta\to 0}0.
	\end{aligned}
	\]
	Moreover,
	\[
	\begin{aligned}
	\psi_{\text{SQO}}(\alpha_\ast,y)&=\frac{1}{\alpha_\ast}\langle p_{\alpha_\ast},p^{II}_{\alpha_\ast} \rangle-\frac{1}{\alpha_\ast}\|p^{II}_{\alpha_\ast}\|^2\le\frac{1}{\alpha_\ast}\|p_{\alpha_\ast}\|\|p^{II}_{\alpha_\ast}\|\le\frac{\|p_{\alpha_\ast}\|^2}{\alpha_\ast}=4\|w\|^2\alpha_\ast\xrightarrow{\delta\to 0}0,
	\end{aligned}
	\]
	and
	\[
	\begin{aligned}
	\psi_{\text{SQO}}(\alpha,y^\delta)&\le\frac{\|p^\delta_\alpha\|^2}{\alpha}\to 0,
	\end{aligned}
	\]
	as $\delta\to 0$ for $\alpha$ chosen appropriately as before. The result then follows.
\end{proof}

\subsubsection{Convergence rates}
For completeness, we provide convergence rates results:
\begin{proposition}
	Let the source condition \eqref{source} hold, $\alpha_\ast$ be selected according to the symmetric quasi-optimality rule and suppose the auto-regularisation condition \eqref{noiseconditionquasiopt} is satisfied. Assume, 
	that $\|\yd\| \geq C$ for all $\delta$ sufficiently small and 
	in addition, that
	\begin{equation}\label{aagx} C \leq \left\langle \frac{\pad}{\alpha} -  
	\frac{\padii}{\alpha} , 	\frac{\padii}{\alpha}\right\rangle\qquad 
	\forall \alpha \in (0,\alpha_\text{max})\qquad\text{and}\qquad \|\yd\| \geq C\end{equation}
	Then
	\[
	D_{\xi}(x^\delta_\alpha,x^\dagger)=\mathcal{O}(\delta),
	\]
	for $\delta>0$ sufficiently small.
\end{proposition}
\begin{proof}
	We have
	\begin{align*}
	D_\xi(x^\delta_{\alpha_\ast},x^\dagger)&\le C\frac{1}{\alpha_\ast}\langle \Dp_\ast-\Dpiiast,\Dpiiast\rangle+C\delta+C\alpha_\ast
	\\
	&=\mathcal{O}\left(\psi_{\text{SQO}}(\alpha,y^\delta)+\psi_{\text{SQO}}(\alpha_\ast,y)+\delta+\alpha_\ast \right)
	\\
	&=\mathcal{O}\left(\left(\frac{\delta}{\sqrt{\alpha}}+C\sqrt{\alpha}\right)^2+\delta+\alpha_\ast\right)=\mathcal{O}\left(\delta^2+\delta+\alpha_\ast\right)=\mathcal{O}(\delta),
	\end{align*}
	where we used that $\alpha_\ast\le\psi_{\text{SQO}}(\alpha_\ast,y^\delta)$, 
	which follows from \eqref{aagx}.
\end{proof}

\subsubsection{The right quasi-optimality rule}
As the expression for the right quasi-optimality functional contains the cumbersome $\mathcal{R}$-functional terms, estimates in which they do not appear may be of utility:
\begin{proposition}
    There exists a positive constant $C$ such that
    \[
    D_{\xi^\delta_\alpha}(x^{II}_{\alpha,\delta},x^\delta_\alpha)\le C\psi_{\text{HD}}(\alpha,y^\delta),
    \]
    for all $\alpha\in(0,\alpha_\text{max})$ and $y^\delta\in Y$.
\end{proposition}
\begin{proof}
We may write
\begin{equation}
\begin{split}
    D_{\xi^\delta_\alpha}(x^{II}_{\alpha,\delta},x^\delta_\alpha)&=\mathcal{R}(x^{II}_{\alpha,\delta})-\mathcal{R}(x^\delta_\alpha)+\frac{1}{\alpha}\langle Ax^\delta_\alpha-y^\delta,Ax^{II}_{\alpha,\delta}-y^\delta-(Ax^\delta_\alpha-y^\delta)\rangle \\
    & 
=\mathcal{R}(x^{II}_{\alpha,\delta})-\mathcal{R}(x^\delta_\alpha)
+ \frac{1}{\alpha}\langle \pad,\padii-\pad\rangle. 
    \end{split} \label{2ndquasineq}
\end{equation}
From the optimality of $x^{II}_{\alpha,\delta}$, we have that
\[
\frac{1}{2}\|Ax^{II}_{\alpha,\delta}-y^\delta-(y^\delta-Ax^\delta_\alpha)\|^2+\alpha\mathcal{R}(x^{II}_{\alpha,\delta})\le 2\|Ax^\delta_\alpha-y^\delta\|^2+\alpha\mathcal{R}(x^\delta_\alpha),
\]
i.e.,
\[
\mathcal{R}(x^{II}_{\alpha,\delta})\le\frac{2}{\alpha}\|Ax^\delta_\alpha-y^\delta\|^2-\frac{1}{2\alpha}\|Ax^{II}_{\alpha,\delta}-y^\delta-(y^\delta-Ax^\delta_\alpha)\|^2+\mathcal{R}(x^\delta_\alpha),
\]
from which we get that
\[
      D_{\xi^\delta_\alpha}(x^{II}_{\alpha,\delta},x^\delta_\alpha)\le\frac{1}{\alpha}\left( 2\|p^\delta_\alpha\|^2-\frac{1}{2}\|p^{II}_{\alpha,\delta}+p^\delta_\alpha\|^2+\langle p^\delta_\alpha,p^{II}_{\alpha,\delta}-p^\delta_\alpha\rangle \right)\le\frac{2}{\alpha}\|p^\delta_\alpha\|^2,
\]
and the desired estimate from above subsequently follows.
\end{proof}
\begin{proposition}
    Let $\alpha_\ast$ be selected according to the right quasi-optimality rule.
    Then $\alpha_\ast\to 0$ as $\delta\to 0$ for all $y^\delta\in Y$.
\end{proposition}
\begin{proof}
Since we can estimate the quasi-optimality functional by the heuristic discrepancy functional, as per the previous proposition, the result follows from \cite{jinlorenz}.
\end{proof}
\begin{theorem}
    Let $x^\dagger$ satisfy \eqref{source},
    $\alpha_\ast$ be selected according to the right quasi-optimality rule,
    and suppose that there exists a positive constant $C>0$ such that
    \begin{equation}
    D_{\xi_\alpha}(x^\delta_\alpha,x_\alpha)\le C D_{\xi^\delta_\alpha}(x^{II}_{\alpha,\delta},x^\delta_\alpha)+\mathcal{O}(\alpha),
    \label{goodqoautoregcond}
    \end{equation}
    holds for all $\alpha\in (0,\alpha_\text{max})$ and $y,y^\delta\in Y$. Then
    \[
    D_{\xi}(x^\delta_{\alpha_\ast},x^\dagger)\to 0,
    \]
    as $\delta\to 0$.
\end{theorem}
We omit the proof as it is analogous to the above. Note that if the source condition \eqref{source} holds, $\alpha_\ast$ is selected according to the right quasi-optimality rule and the auto-regularisation condition \eqref{goodqoautoregcond} is satisfied, then one may also prove that
\[
D_\xi(x^\delta_\alpha,x^\dagger)=\mathcal{O}(\delta),
\]
for $\delta>0$ sufficiently small, provided that $\alpha_\ast\le\psi_{\text{RQO}}(\alpha_\ast,y^\delta)$.

\begin{remark}
Observe that the heuristic discrepancy, the Hanke-Raus, and the 
symmetric quasi-optimality rules can all be expressed in terms of 
the residuals of the Bregman iteration $\pad,\padii$. It should be 
noted that in the linear case  they
can all be subsumed under the so-called family of 
R1-rules \cite{RausR1}. The similarity of patterns in the formulas for $\psi$ 
may provide a hint that such a larger family of rules could be 
defined in the convex case as well. 
\end{remark}

\section{Diagonal operator}\label{DiagSec}
In the following analysis, we consider the case in which the operator $A:X\to Y$ is diagonal between spaces of summable sequences; in particular, where $X=\ell^q(\mathbb{N})$ and $Y=\ell^2(\mathbb{N})$ and the regularisation functional is selected as the $\ell^q$-norm to the $q$-th power. 
The main objective in this setting is 
to further investigate the auto-regularisation conditions  and to illustrate their validity for specific instances.

Let $\{e_n\}_{n\in\mathbb{N}}$ be a basis for $X$ and let $\{\lambda_n\}_{n\in\mathbb{N}}$ be a sequence of real scalars
monotonically decaying to $0$. Then we define a diagonal operator $A:\ell^q(\mathbb{N})\to\ell^2(\mathbb{N})$,
	\begin{equation*}
	Ae_n=\lambda_ne_n %\qquad\text{(cf.~\cite{conway})}.
	\end{equation*}
The regularisation functional is chosen as
\[
\mathcal{R}:=\frac{1}{q}\|\cdot\|^q_{\ell^q}\qquad\text{and}\qquad \partial\mathcal{R}(x)=\{|x_n|^{q-1}\operatorname{sgn}(x_n)\}_{n\in\mathbb{N}},\qquad q\in(1,\infty).
\]
In this situation, the operator decouples and the components of the regularised solution can be computed independently of each other. Thus, for notational purposes, we opt to omit the sequence index $n$ for the components of 
the regularised solutions and write 
	\[
	x^\delta_\alpha=: \{x^\delta_{\alpha,n}\}_n =:\{\xadn\}_n,\quad  
	x_\alpha=: \{x_{\alpha,n}\}_n =:\{\xan\}_n,\quad
	\yd =:\{\ydn\}_n,\quad y=:\{\yn\}_n, \]
%	\qqua
where $\xan, \xadn,\ydn,\yn \in \mathbb{R}$.
As the problem decouples, $\xan$ and $\xadn$ can be computed by an optimisation 
problem on $\mathbb{R}$, i.e., the optimality conditions 
lead to 
	\[
	\xadn =\hqi{q,\gamma_n}\left(\frac{\ydn}{\lambda_n}\right)\qquad\text{and}\qquad\xadiin:=x^{II}_{\alpha_n,\delta_n}=\hqi{q,\gamma_n} \left(2\frac{\ydn}{\lambda_n}-\xadn\right),\qquad\text{with }\gamma_n:=\frac{\alpha}{\lambda^2_n},
	\]
	for all $\alpha\in(0,\alpha_{\text{max}})$ and $\ydn\in\mathbb{R}$.
Here $\hq{q,\gamma}$ is the inverse of the function $\hq{q,\gamma}: \mathbb{R}\to  \mathbb{R}$,
\[ \hq{q,\gamma_n}(x)  = x +\gamma_n |x|^{q-1} \sgn(x), \qquad x \in \mathbb{R}.\]
Note that $\hqi{q,\gamma_n}$ corresponds to a proximal operator on $\mathbb{R}$.

%\subsection{The convergence condition}
Define $\padn:=\ydn-\lambda_n \xadn$ and $\pan$ analogously in case of exact data;
furthermore we use the expressions $\yn$, $\ydn$,  $\Dyn$, $\Dpn$,  $\Dpiin$ 
to denote the components of $\yd,$ $y,$ $\Dy,$ $\Dp$, $\Dpii$, respectively, 
where we again omit the sequence index $n$ in the notation:
\[
\ydn = y^\delta_n, \quad \yn = y_n, \quad 
\Dyn:= y_n-y^\delta_n, \qquad  \Dpn = \pa_n-{p_{\alpha}^\delta}_n, \qquad 
\Dpiin: = p^{II}_{\alpha_n}- \padii, \qquad n \in \mathbb{N}.
\]
For the $\ell^q$-case we can apply an appropriate scaling to reduce the corresponding inequalities:
in fact, letting 
\[  \hq{q}(x):= x + |x|^{q-1} \mbox{sign}(x), \qquad x \in \mathbb{R}, 
\qquad \text{and } \quad   \etan :=  \gamma_n^{\tfrac{1}{2-q}},
\] 
%then, with 
%\[ \]
we obtain
\begin{align} 
\xadn &= \etan \hqi{q}(\tfrac{\ydn}{\etan \lambda_n}), \\
\padn & = \lambda_n \etan  \hqi{q^*}(\tfrac{\ydn}{\etan \lambda_n}),  \label{padH}\\
\xadiin & = \etan  \hqi{q}\left(\tfrac{\ydn}{\etan \lambda_n},
+  \hqi{q^*}(\tfrac{\ydn}{\etan \lambda_n})\right),\\
\padiin & = 
\lambda_n \etan\left(  \hqi{q^*}\left(
\tfrac{\ydn}{\etan \lambda_n} + 
 \hqi{q^*}(\tfrac{\ydn}{\etan \lambda_n})\right) 
-  \hqi{q^*}(\tfrac{\ydn}{\etan \lambda_n}) \right).
\end{align}
Note that for $x>0$ we have 
\begin{equation}\label{h1}  x^{q-1} \leq \hq{q}(x).  \end{equation}
We now state  useful estimates for the function $\hqi{q}$:
\begin{lemma}\label{usefullemma}
For $1 <q<\infty$, $q \not = 2$, there exist  constants $\dlo,D_q,$ and for any $\tau >0$, a constant $D_{q,\tau}$ such that 
for all $x_1 >0$ and $|x_2| \leq x_1$, 
\begin{align} \label{mmx}
\frac{1}{1 + \dlo  \hqi{q}(x_1)^{q-2}} \leq \frac{\hqi{q}(x_1) - \hqi{q}(x_1)}{x_1-x_2} \leq  \frac{1}{1 + \dup  \hqi{q}(x_1)^{q-2}} 
\leq \begin{dcases} 
\frac{x_1^{\tfrac{2-q}{q-1}}}{D_q}   & \mbox{ if } 1 < q < 2, \\ 
\frac{x_1^{\tfrac{2-q}{q-1}}}{D_{q,\tau} }  &\mbox{ if }   q>2 \mbox{ and } \forall  x_1 >\tau.  \\ 
\end{dcases} 
\end{align}
\end{lemma} 
\begin{proof}
For any $z_1>0, |z_2| \leq z_1$, we have
\[\frac{\hq{q}(z_1)-\hq{q}(z_2)}{z_1-z_2}=1+\frac{z_1^{q-1}-|z_2|^{q-1}\sgn(z_2)}{z_1-z_2}=
			1+z_1^{q-2}k\left(\frac{z_2}{z_1}\right) \quad \begin{cases} 
			\leq 1+z_1^{q-2} \dlo \\ 
			\geq 1 + z_1^{q-2} \dup \end{cases}
\]
where 
\[
			\dup\le k(z):=\frac{1-|z|^{q-1}\sgn(z)}{1-z}\le\dlo, \qquad \forall z \in [-1,1].
\]
Replacing $z_i$ by $\hqi{q}(x_i)$ yields the lower bound and the first upper bound  in  \eqref{mmx}.
In case $1 < q <2$, we find that 
\[  \frac{1}{1 + \dup  \hqi{q}(x_1)^{q-2}} = 
 \frac{ \hqi{q}(x_1)^{2-q}}{ \hqi{q}(x_1)^{2-q} + \dup } \leq 
  \frac{ x_1^\frac{2-q}{q-1}}{ \dup },  \]
where we used that  $\hqi{q}(x_1)^{2-q} \geq 0$ in the denominator and the estimate 
 $\hqi{q}(x_1) \leq x_1^\frac{1}{q-1}$ that follows from \eqref{h1}.
 Now consider the case $q>2$. Then 
 \[ \frac{1}{1 + \dup  \hqi{q}(x_1)^{q-2}} \leq \frac{C_\tau^{\tfrac{2-q}{q-1}}}{\dup} x_1^{\tfrac{2-q}{q-1}} \qquad \forall x_1 \geq \tau, \]
 where we used the estimate 
 \[ \hq{q}(x) \leq C_\tau x^{q-1} \qquad \forall x \geq \hq{q}(\tau).\] 
This yields the result. 
\end{proof}

\subsubsection{The heuristic discrepancy rule}
In case the forward operator is diagonal, we may reduce the auto-regularisation condition to Muckenhoupt-type inequalities \cite{kindermannpilipenko} which may shed more light on them.
\begin{proposition}\label{prophd}
	Let $A:\ell^q(\mathbb{N})\to\ell^2(\mathbb{N})$ be a diagonal operator and let $\mathcal{R}=\frac{1}{q}\|\cdot\|^q_{\ell^q}$. If 
	\begin{equation}
	\sum_{n\in I_{\text{HD}}^c} \realp{\Dpn}{\Dyn}  \le C \sum_{n\in I_{\text{HD}}}|\Dyn|^2,
	\label{hdmuckenhouptineq}
	\end{equation}
	is satisfied, where,
	\[
	I_{\text{HD}}:=\left\{n\in\mathbb{N}:|\Dyn|\le C|\Dpn| \right\}.
	\]
	Then there exists a positive constant $C>0$ such that
	\begin{equation}
	D_{\xi_\alpha}(x^\delta_\alpha,x_\alpha)\le C\frac{\|\Dp\|^2}{\alpha},
	\label{autoregularisationhd}
	\end{equation}
	for all $\alpha\in(0,\alpha_\text{max})$ and $y,y^\delta\in Y$.
	\end{proposition}
	\begin{proof}
		For $A:\ell^q(\mathbb{N})\to\ell^2(\mathbb{N})$ a diagonal operator, \eqref{hdabstractineq} may be rewritten as 
		\[
		\begin{aligned}
		&\sum_{n\in\mathbb{N}}\realp{\Dpn}{\Dyn}\le C\sum_{n\in\mathbb{N}}|\Dpn|^2.
		\end{aligned}
		\]
		Now, we may decompose the above sum as the superposition of two sums to obtain
		\[
		\left(\sum_{n\in I_{\text{HD}}}+\sum_{n\in I_{\text{HD}}^c}\right)
		\realp{\Dpn}{\Dyn} \le C\sum_{n\in I_{\text{HD}}}|\Dpn|^2+\sum_{n\in I_{\text{HD}}^c} 	\realp{\Dpn}{\Dyn}.
		\]
		Thus we may obtain the desired inequality.
		\end{proof}
We are now in the position to state the main convergence result for the heuristic discrepancy principle. 
\begin{theorem}\label{th:six}
Let $A$ and $\mathcal{R}$ be as in Proposition~\ref{prophd}
and $q\in (1,\infty)$. 
Let $\alpha_*$ be selected by the heuristic discrepancy rule, 
let the source condition \eqref{source} be satisfied, $\delta^\ast\ne 0$ and
suppose that there are constants $C_1$, $C_2$ such that for all $\yd$ it holds that 
	\begin{equation}
	\sum_{\{ n: \tfrac{\lambda_n^q}{\alpha} \max\{|\yn|,|\ydn|\}^{2-q}\geq C_1\}} 
	|\Dyn|^2
	\max\{|\yn|,|\ydn|\}^{q-2} \frac{\alpha}{\lambda_n^q} \leq 
	C_2 
		\sum_{\{n: \tfrac{\lambda_n^q}{\alpha} \max\{|\yn|,|\ydn|\}^{2-q}< C_1\}} 
          |\Dyn|^2.
	\label{hdmuckenhouptineqsimple}
	\end{equation}
Then 
\[ D_\xi(x^\delta_{\alpha_\ast},x^\dagger)\to 0\qquad \text{as } \delta\to 0. \]
\end{theorem}
\begin{proof}
From the preceding results, it remains to verify that \eqref{hdmuckenhouptineqsimple} implies 
\eqref{hdmuckenhouptineq}. For this, we estimate the ratio $\frac{\Dpn }{\Dyn}$.
We note that by monotonicity, 
this expression is always positive, and moreover, is invariant when $\yn$, $\ydn$
are switched and when $\yn$, $\ydn$ are replaced by $-\yn$, $-\ydn$. Thus, without loss 
of generality we may assume that $\yn >0$ and $|\ydn| \leq \yn$. Using \eqref{padH}, 
Lemma~\ref{usefullemma} with $x_1  = \frac{\yn}{\lambda_n \etan} =  \frac{\max\{|\yn|,|\ydn|\}}{\lambda_n \etan}$ and
$x_2 =   \frac{\ydn}{\lambda_n \etan}$, we 
find that 
\begin{equation}\label{anix}
\frac{\Dpn }{\Dyn} \geq 
\frac{1}{1 + \dlo  \hqi{q^*}\left(\frac{\yn}{\lambda_n \etan}\right)^{q^*-2}}. 
\end{equation} 
A brief consideration yields that \eqref{hdmuckenhouptineq} is satisfied if the corresponding 
inequality is satisfied when $I_{\text{HD}}$ is replaced with any $I_{\text{HD}}' \subset I_{\text{HD}}$ (on the left and right-hand side of
the inequality). Thus, a sufficient inequality is when the index set 
\[ I_{\text{HD}}':= \left\{ n \in \mathbb{N} : \frac{1}{1 + \dlo  \hqi{q^*}\left(\frac{\yn}{\lambda_n \etan}\right)^{q^*-2}} \geq \frac{1}{C} \right\},  \]
is used in place of $I_{\text{HD}}$. The condition in this index set can be simplified to 
\[     \frac{ \max\{|\yn|,|\ydn|\}}{\lambda_n \etan}   \begin{cases} \leq  \hq{q^*}\left(\left(\frac{C-1}{\dlo}\right)^\frac{1}{q^*-2}\right),    
& q^*>2, \\ 
 \geq  \hq{q^*}\left(\left(\frac{C-1}{\dlo}\right)^\frac{1}{q^*-2}\right),
& 1<q^*<2.  \end{cases} 
%\hq\left( (\frac{1}{C}-1)^\frac{1}{q^*-2}\right)  & 
\] 
Note that $\lambda_n \etan = \left(\frac{\alpha}{\lambda^q}\right)^\frac{1}{2-q}$. Thus, in any case, we obtain that 
\[ I_{\text{HD}}':= \{ n \in \mathbb{N} : \tfrac{\lambda_n^q}{\alpha} \max\{|\yn|,|\ydn|\}^{2-q} \leq C_q  \}, \]
with $C_q =  \left[\hq{q^*}\left(\left(\frac{C-1}{\dlo}\right)^\frac{1}{q^*-2}\right) \right]^{2-q}$.
In the next step, we observe that Lemma~\ref{usefullemma} yields the upper bound 
\begin{equation}\label{thi} \frac{\Dpn }{\Dyn} \leq 
C \left[\max\{|\yn|,|\ydn|\}^{2-q} \frac{\lambda^q}{\alpha}\right]^{\frac{1}{2-q}\frac{2-q^*}{q^*-1}},
\end{equation}
in case that $1 \leq q^* \leq 2$.  The same estimate holds for $q^*>2$ 
whenever 
$ \frac{ \max\{|\yn|,|\ydn|\}}{\lambda_n \etan}   \geq \tau$. However, as the left-hand side 
is a sum over the complement of 
${I_{\text{HD}}'}$, the latter condition holds at the complement; hence \eqref{thi} is true in any case. 
As $\frac{2-q^*}{q^*-1} = q-2$, and since  the condition with an upper bound for 
the left-hand side  is sufficient 
for \eqref{hdmuckenhouptineq}, it is thus shown that \eqref{hdmuckenhouptineqsimple}  implies \eqref{hdmuckenhouptineq}.

\end{proof} 
\subsubsection{The Hanke-Raus rule}
The following are sufficient conditions for the auto-regularisation condition of the Hanke-Raus rule to hold in the present setting.
\begin{proposition}\label{prophr}
	Let $A:\ell^q(\mathbb{N})\to\ell^2(\mathbb{N})$ be a diagonal operator and 
	let $\mathcal{R}=\frac{1}{q}\|\cdot\|^q_{\ell^q}$. 
Suppose that for all $\yn,\ydn$ and all $n \in \mathbb{N}$,
\begin{equation}\label{posHR}
\realp{\Dpiin}{\Dpn} \geq 0. \end{equation}	
Furthermore, let the condition
	\begin{equation}
	\sum_{n\in I_{\text{HR}}^c} \realp{\Dpn}{\Dyn} -  |\Dpn|^2	
	\le C\left(  \sum_{n\in I_{\text{HR}}}\realp{\Dyn}{\Dpn} \right)
	\label{hrmuckenhoupt}
	\end{equation}
	hold, where
	\[
	I_{\text{HR}}:=%I_{\text{HD}}\cap
	\left\{n\in\mathbb{N}: \theta  
	|\Dpn +\Dyn|\le |\Dpiin+ 
	\Dpn| \right\}, \qquad \mbox{ for some } \theta > \frac{1}{2}. 
	\]
	Then there exists a positive constant $C>0$ such that
	\[
	\begin{aligned}
	D_{\xi_\alpha}(x^\delta_\alpha,x_\alpha)&\le C \frac{1}{\alpha}\langle \Dpii,\Dp\rangle,
	\end{aligned}
	\]
	for all $\alpha\in(0,\alpha_\text{max})$.
\end{proposition}
\begin{proof}
	Recall that it suffices to prove that
	\begin{equation}\label{scd}
	\sum_{n\in\mathbb{N}}
	\realp{\Dpn}{\Dyn}
	-\sum_{n\in\mathbb{N}}|\Dpn|^2
	\le C\sum_{n\in\mathbb{N}}
	\realp{\Dpiin}{\Dpn}.	
	\end{equation}
We may define 
\[ R^{II}:= \frac{\Dpiin+ \Dpn}
{\Dpn +\Dyn}, \]
and observe that this quantity is in any case nonegative and for $n \in I_{\text{HR}}$ larger than $\theta >\frac{1}{2}$. 
Thus, 
\begin{align*} 
&\realp{\Dpiin}{\Dpn}= 
 R^{II}\left(|\Dpn|^2 +  \Dyn \Dpn\right) - 
|\Dpn|^2\geq 
  \theta  \Dyn \Dpn - (1-\theta)  |\Dpn|^2  
 \geq   (2 \theta -1 )\Dyn  \Dpn.
%  %
% &\langle (p^{II}_{\alpha_n}-p^{II}_{\alpha_n,\delta_n}+ p_{\alpha_n}+p^{\delta_n}_{\alpha_n}),p_{\alpha_n}-p^{\delta_n}_{\alpha_n}\rangle -
% |p_{\alpha_n}-p^{\delta_n}_{\alpha_n}|^2 \\
% & =  \frac{p^{II}_{\alpha_n}-p^{II}_{\alpha_n,\delta_n}+ p_{\alpha_n}+p^{\delta_n}_{\alpha_n}}
% {p_{\alpha_n}-p_{\alpha_n,\delta_n} +y_n-y^{\delta_n}} (p_{\alpha_n}-p_{\alpha_n,\delta_n} +
% y_n-y^{\delta_n})(p_{\alpha_n}-p^{\delta_n}_{\alpha_n})  - |p_{\alpha_n}-p^{\delta_n}_{\alpha_n}|^2 \\
% & \geq_{sign!}   \frac{1}{2} (p_{\alpha_n}-p_{\alpha_n,\delta_n} +
% y_n-y^{\delta_n})(p_{\alpha_n}-p^{\delta_n}_{\alpha_n})  - |p_{\alpha_n}-p^{\delta_n}_{\alpha_n}|^2 = 
%   \frac{1}{2} (y_n-y^{\delta_n})(p_{\alpha_n}-p^{\delta_n}_{\alpha_n}) -   \frac{1}{2} |p_{\alpha_n}-p^{\delta_n}_{\alpha_n}|^2 
\end{align*} 
Therefore, \eqref{scd} holds for $n \in I_{\text{HR}}$. The remaining sum can be bounded by 
 \eqref{hrmuckenhoupt} and because of \eqref{posHR}, the 
 sum over $I_{\text{HR}}^c$ on the right-hand side can 
 be estimated by $0$ from below. This suffices to prove the statement. 
\end{proof}
 
We proceed by verifying \eqref{posHR}. Contrary to the heuristic 
discrepancy case, we have to impose a restriction on the 
regularisation functional $\mathcal{R}=\frac{1}{q}\|\cdot\|^q_{\ell^q}$.
\newcommand{\vv}{r} 
\begin{lemma}
If $q \geq \frac{3}{2}$, then it follows that
\[ \realp{\Dpiin}{\Dpn}   \geq 0, \]
for all $\alpha\in(0,\alpha_\text{max})$ and $y^\delta\in Y$.
\end{lemma} 
\begin{proof}
Setting $z_1 = \frac{\yd}{\etan \lambda_n}$, $z_2 = \frac{y}{\etan \lambda_n}$, the statement is verified if we can show that 
for any $z_1,z_2$,
\[ 
\left[\hqi{q^*}\left(z_1 + \hqi{q^*}(z_1)\right) - \hqi{q^*}(z_1) -  \left(  \hqi{q^*}\left(z_2 + \hqi{q^*}(z_2)\right)- 
 \hqi{q^*}(z_2) \right) \right] \left[ \hqi{q^*}(z_1) - \hqi{q^*}(z_2)\right]  \geq 0. \] 
Since $\hqi{q^*}\left(z_1 + \hqi{q^*}(z_1)\right) = \hqi{q^*}\left(\hq{q^*}( \hqi{q^*}(z_1) )+ \hqi{q^*}(z_1)\right)$,
it is enough to show that the mapping (note the one-to-one correspondence of $z_i \mapsto \hqi{q^*}(z_i)$), 
\[ F: p \mapsto \hqi{q^*}\left(\hq{q^*}(p)+ p\right) - p , \] 
is monotonically increasing. As this function is differentiable everywhere except at $p =0$, it suffices to prove the inequality 
\[ 0 \leq F'(p) = \frac{2+ (q^*-1)|p|^{q^*-2} }{1 +(q^*-1)  | \hqi{q^*}\left(\hq{q^*}(p)+ p\right)|^{q^*-2}}-1, \] 
for any $p \in \mathbb{R}$. Since $F$ is antisymmetric and hence $F'$ is symmetric, it is sufficient to 
prove this inequality for $p >0$.  Setting $\vv = \hqi{q^*}\left(\hq{q^*}(p)+ p\right)$, we thus have to show that 
\[ \frac{2+ (q^*-1)|p|^{q^*-2} }{1 +(q^*-1)  |\vv|^{q^*-2}}\geq 1,  \qquad\mbox{where }  \hq{q^*}(\vv) = \hq{q^*}(p)+ p,\] 
which reduces to 
\[ 1 + (q^*-1)p^{q^*-2}  \geq (q^*-1)  \vv^{q^*-2} \qquad  \forall \vv>0,p>0, \mbox{ with } \vv + \vv^{q-1} = 2 p + p^{q-1}. \]
Some algebraic manipulation allows us to express $p^{q-2}$ in terms of $\zeta := \frac{\vv}{p}$ as 
\[ p^{q-2} = \frac{2 -\zeta}{\zeta^{q-1}-1}.  \] 
The right-hand side is monotonically decreasing for $\zeta \in [1,2]$, thus we may invert the equation and 
express $\zeta$ as function of $p$. Setting 
$\vv^{q^*-2} = p^{q^*-2} \zeta^{q^*-2}$, the  
inequality we would like to prove is then 
\begin{equation}\label{ube}     (q^*-1)  \frac{(2 -\zeta)(\zeta^{q^*-2} -1)}{\zeta^{q-1}-1}               \leq 1, \qquad \forall \zeta \in [1,2].\end{equation}
For $q^*\leq 2$, the left-hand side is clearly negative for $\zeta \in [1,2]$, and thus the inequality is trivially 
satisfied. For $q^*>2$,  it  can be verified 
by a detailed analysis that for 
$q^*-1\leq 2$,
 \[  (q^*-1)  \frac{(\zeta^{q^*-2} -1)}{\zeta^{q-1}-1} \leq 1,\]
 which proves \eqref{ube}.
 Note that the condition $q^* \leq 3 $ is equivalent to 
 $q \geq \frac{3}{2} $.
\end{proof}

% 
% 
% 
% 	Then it follows that splitting the sum
% 	\begin{align*}
% 	\left(\sum_{n\in I_{\text{HR}}}+\sum_{n\in I_{\text{HR}}^c} \right)\langle p_{\alpha_n}-p^{\delta_n}_{\alpha_n},y_n-y^{\delta_n}\rangle&\le C\sum_{n\in I_{\text{HR}}}\langle p^{II}_{\alpha_n}-p^{II}_{\alpha_n,\delta_n},p_{\alpha_n}-p^{\delta_n}_{\alpha_n}\rangle+\sum_{n\in I_{\text{HR}}}|y_n-y^{\delta_n}|^2
% 	\\
% 	&\qquad+\sum_{n\in I_{\text{HR}}^c}\langle p_{\alpha_n}-p^{\delta_n}_{\alpha_n},y_n-y^{\delta_n}\rangle,
% 	\end{align*}
% 	and postulating the Muckenhoupt type noise condition in the statement of the theorem yields the desired inequality.
%\end{proof}
\begin{theorem}\label{th:seven}
Let $A$ and $\mathcal{R}$ be as in Proposition~\ref{prophd}
with $\frac{3}{2}\leq q <\infty$. 
Let $\alpha_*$ be selected by the Hanke-Raus functional, 
let the source condition \eqref{source} be satisfied, $\delta^\ast\ne 0$ and
suppose that there  is a  constant $C_1$, which is sufficiently small  
and a constant 
$C_2$ such that for all $\yd$, it holds that 
	\begin{equation}
	\sum_{\{ n: \tfrac{\lambda_n^q}{\alpha} \max\{|\yn|,|\ydn|\}^{2-q}\geq C_1\}} 
	|\Dyn|^2
	\max\{|\yn|,|\ydn|\}^{q-2} \frac{\alpha}{\lambda_n^q} \leq 
	C_2 
		\sum_{\{n: \tfrac{\lambda_n^q}{\alpha} \max\{|\yn|,|\ydn|\}^{2-q}< C_1\}} 
          |\Dyn|^2.
	\label{hrmuckenhouptineqsimple}
	\end{equation}
Then 
\[ D_\xi(x^\delta_{\alpha_\ast},x^\dagger)\to 0\qquad \text{as } \delta\to 0. \]
\end{theorem}
\begin{proof}
We verify that 	\eqref{hrmuckenhouptineqsimple} implies  \eqref{hrmuckenhoupt}. 
Again with similar considerations as before we may assume $|\ydn|\leq \yn$ and use 
$$x_1 = \tfrac{\yn}{\etan \lambda_n} + \hqi{q^*}(\tfrac{\yn}{\etan \lambda_n}), \qquad 
x_2 = \tfrac{\ydn}{\etan \lambda_n} + \hqi{q^*}(\tfrac{\ydn}{\etan \lambda_n}). $$
Then 
\begin{align*}
 \frac{\Dpiin+ \Dpn}
{\Dpn +\Dyn} 
= \frac{\hqi{q^*}(x_1) - \hqi{q^*}(x_2)}{x_1 -x_2} \geq 
\frac{1}{1 + \dlo  \hqi{q^*}(x_1)^{q^*-2}}.
\end{align*}
As before we may define a new index set $I_{\text{HR}}'$ of indices, where 
\[ \frac{\max\{|\yn +  \pan|,|\ydn +  \padn|\}}{\lambda_n \etan} 
\begin{cases} \leq  \hq{q^*}\left(\left(\frac{\frac{1}{\theta}-1}{\dlo}\right)^\frac{1}{q^*-2}\right),    
& q^*>2, \\ 
 \geq  \hq{q^*}\left(\left(\frac{\frac{1}{\theta}-1}{\dlo}\right)^\frac{1}{q^*-2}\right),   
& 1<q^*\leq2,  \end{cases} 
\]
such that 
\[ I_{\text{HR}}' = 
 \{ n \in \mathbb{N} : \tfrac{\lambda_n^q}{\alpha}\max\{|\yn +  \pan|,|\ydn +  \padn|\}^{2-q} \leq C_q  \}, \]
with $C_q =  \left[\hq{q^*}\left(\left(\frac{\frac{1}{\theta} -1}{\dlo}\right)^\frac{1}{q^*-2}\right) \right]^{2-q}$.
Since $ 0 \leq \pan \leq |\yn|$ (and similar for the $\delta$-part), we obtain 
\[ \max\{|\yn +  \pan|,|\ydn +  \padn|\}^{2-q} \leq \begin{cases} 
2^{2-q} \max\{|\yn|,|\ydn\}^{2-q},  & 1 < q <2, \\ 
 \max\{|\yn|,|\ydn\}^{2-q}, & q>2, \end{cases} \]
which yields a subset
\[ I_{\text{HR}}'' =  \{ n \in \mathbb{N} : \tfrac{\lambda_n^q}{\alpha}\max\{|\yn|,|\ydn|\}^{2-q} \leq C_q'  \}, \]
where  $C_q' = \max\{2^{2-q},1\} C_q$.  Note that the restriction $\theta >\frac{1}{2}$ can always 
be achieved by postulating that $C_q'$ is sufficiently  small. 
We observe that except for the constant $C_q'$, $I_{\text{HR}}''$ agrees with $I'_{\text{HD}}$; hence we may 
estimate the right-hand side from below by 
$C \sum_{n \in I_{\text{HR}}''} |\Dyn|^2$.  On the other hand, the left-hand side can 
be bounded from above exactly the same as for the heuristic discrepancy rule, which 
shows the sufficiency of 	\eqref{hrmuckenhouptineqsimple} for  \eqref{hrmuckenhoupt}. 
\end{proof} 
We remark that apart from the smallness condition on $C_1$, this is the same condition as for the 
heuristic discrepancy rule. This completely agrees with the linear theory ($q = 2$).

\subsubsection{The symmetric quasi-optimality rule}
%\paragraph{The \textquotedblleft symmetric\textquotedblright quasi-optimality rule}

\begin{proposition}\label{propqo}
	Let $A:\ell^q(\mathbb{N})\to\ell^2(\mathbb{N})$ be a diagonal operator and 
	let $\mathcal{R}=\frac{1}{q}\|\cdot\|^q_{\ell^q}$. 
Suppose that for all $\yn,\ydn$ and all $n \in \mathbb{N}$
\begin{equation}\label{posQO}
\realp{\Dpiin}{(\Dpn - \Dpiin)} \geq 0. \end{equation}	
Furthermore, let the condition
	\begin{align}\label{muckqo}
	\sum_{n\notin I_{\text{SQO}}}\realp{(\Dpn-\Dyn)}{\Dpn}
	 \leq  C \sum_{n\in I_{\text{SQO}}}\realp{(\Dpn-\Dpiin)}{\Dpiin},	\end{align}
	hold, where for some constant $C$
	\begin{align*} 
	I_{\text{SQO}}:=\big\{n\in\mathbb{N}:  |\Dpn-\Dyn| &\leq 
	C |\Dpn- \Dpiin|  \qquad \mbox {and } \\
	 |\Dpn|&\leq C |\Dpiin| \big\}.
	\end{align*}
	Then there exists a constant such that
	\[
	D_{\xi_\alpha}(x^\delta_\alpha,x_\alpha)\le C\frac{1}{\alpha}\langle \Dp-\Dpii,\Dpii\rangle,
	\]
	for all $\alpha\in(0,\alpha_\text{max})$ and $y^\delta\in Y$.
\end{proposition}
\begin{proof}
By Lemma~\ref{Lemma1}, it suffices to show that
\begin{equation}\label{blahqo}
\langle \Dp, \Dy-\Dp\rangle\le C\langle \Dp-\Dpii,\Dpii\rangle.
\end{equation}
By the nonnegativity of \eqref{posQO}
and monotonicity, we have 
\[ \realp{(\Dpn-\Dyn)}{
	\Dpn}
	 = 
	|\Dpn-\Dyn|
	|\Dpn|,
\]
and	
\[ \realp{(\Dpn-\Dpiin)}
	{\Dpiin}
	= 
	|\Dpn-\Dpiin|
	|\Dpiin|, \]
such that by the definition of 	$I_{\text{SQO}}$ 
the sum  over $I_{\text{SQO}}$ on the left-hand side of \eqref{blahqo} can be bounded 
by that on the right-hand side. The sum over the complement  of 
 $I_{\text{SQO}}$ on the right can be bounded by \eqref{posQO}
 from below by $0$ and the corresponding sum on the left is bounded 
 by condition \eqref{muckqo}, which proves the statement. 
\end{proof}

Similar as for the Hanke-Raus rule, we first have to verify 
the nonnegativity of certain expressions: 
\begin{lemma}
If $q\geq \frac{3}{2}$, then
\[ \realp{(\Dpn-\Dpiin)}{\Dpiin} \geq 0, 
	\]
for all $\alpha\in(0,\alpha_\text{max})$ and $\yn,\ydn\in \mathbb{R}$.	
\end{lemma} 
\begin{proof}
Recall the mapping $F: \pan \mapsto \paiin$ defined above. 
In order to prove monotonicity, it is enough to show that  
\[  \left((p_1 - F(p_1))  - (p_2 - F(p_2) \right)( F(p_1)-  F(p_2)) \geq 0, \qquad \forall p_1, p_2, \] 
or 
\[  \left((p_1-p_2\right)( F(p_1)-  F(p_2)) \geq  |F(p_1)-  F(p_2))|^2,\qquad \forall p_1, p_2. \] 
Thus, if $F$ would be  monotone and Lipschitz continuous, then the inequality would hold. 
Thus, we require the condition that 
\[ 0 \leq F'(p) \leq 1, \qquad \forall p. \] 
In the notation of above, this means that 
\[ \begin{array}{lcrr} 
1 + (q^*-1)p^{q^*-2}  &\geq& (q^*-1)  \vv^{q^*-2} &\mbox{and} \\[3mm] 
(q^*-1)p^{q^*-2}  &\leq& 2 (q^*-1)  \vv^{q^*-2} \end{array} 
\qquad  \forall \vv>0\text{ and }p>0 \mbox{ with } \vv + \vv^{q-1} = 2 p + p^{q-1}. \]
In terms of $\zeta \in [1,2]$ this means that, in addition to the condition for the positivity of the Hanke-Raus rule, we require that
\[ \zeta^{q^*-2}\geq \frac{1}{2}. \]
In particular, this inequality is satisfied for any $q^*\geq 1$, which shows the result. 
\end{proof}

\begin{theorem}\label{th:eight}
Let $A$ and $\mathcal{R}$ be as in Proposition~\ref{propqo} with 
$\frac{3}{2}\leq q <\infty$. 
 Let $\alpha_*$ be selected by the symmetric quasi-optimality 
rule, 
let the source condition \eqref{source} be satisfied, $\delta^\ast\ne 0$ and
suppose that there  is a  constant $C_1$, which is sufficiently small  and  constants 
$C_2, C_3$ such that for all $\ydn$, it holds that 
	\begin{equation}
	\begin{split}
	\sum_{\{ n: \tfrac{\lambda_n^q}{\alpha} \max\{|\yn|,|\ydn|\}^{2-q}\geq C_1\}} 
	|\Dyn|^2
	\max\{|\yn|,|\ydn|\}^{q-2} \frac{\alpha}{\lambda_n^q} 
	+ 
		\sum_{\{ n: \tfrac{\lambda_n^q}{\alpha} \max\{|\yn|,|\ydn|\}^{2-q}\leq C_1\} \cap I_{2}^{c}} 
		          |\Dyn|^2	\\
	\leq 	C_2 
		\sum_{\{n: \tfrac{\lambda_n^q}{\alpha} \max\{|\yn|,|\ydn|\}^{2-q}< C_1\}\cap I_{2}} 
         \left[\frac{\lambda_n^q \max\{|\yn|,|\ydn|\}^{2-q}}{\alpha} \right]^\frac{1}{q-1} |\Dyn|^2,
          \end{split} 
          % ****   means loer bound for 
          %(\Delta p - \Delta y)\Delta p 
%
	\label{qrrmuckenhouptineqsimple}
	\end{equation}
where for some constant $C_3$,  
\[ I_2 =\big\{n\in\mathbb{N}:  |\Dpn-\Dyn| \leq 
	C_3 |\Dpn-\Dpiin|  \big \}.  \]
Then 
\[ D_\xi(x^\delta_{\alpha_\ast},x^\dagger)\to 0\qquad \text{as } \delta\to 0. \]
\end{theorem}
\begin{proof} 
Similar as for the Hanke-Raus rule, 
the second condition in the definition of $I_{\text{SQO}}$,
\[	 |\Dpn|\leq C |\Dpiin|,
\]
holds true if  
\begin{equation}\label{xqz} 
\tfrac{\lambda_n^q}{\alpha} \max\{|\yn|,|\ydn|\}^{2-q}< C_1.
\end{equation}
The sum over this set on the left-hand side  of  
\eqref{muckqo} can be bounded in the same way as before 
by the first sum in \eqref{qrrmuckenhouptineqsimple}. 
The second sum is an upper bound for 
the sum on the left-hand side  in the complement 
of $I_{\text{SQO}}$, where \eqref{xqz} is not satisfied. 
On $I_{\text{SQO}}$, we have the estimate that 
\begin{align*}
&\realp{(\Dpn-\Dpiin)}{
	\Dpiin} \geq 
    |\Dpn-\Dyn|
	|\Dpn|
	\\ 
&\geq \realp{(\Dpn-\Dyn)}{\Dyn}
\geq |\Dyn|^2 \frac{\Dpn}
{\Dyn} \geq C 
 \hqi{q^*}\left(\frac{\max\{|\yn|,|\ydn|\}}{\lambda_n \etan}\right)^{q^*-2},\\
 & = 
  C \hqi{q^*}\left(\left[\frac{\lambda_n^q \max\{|\yn|,|\ydn|\}^{2-q}}{\alpha} \right]^\frac{1}{2-q}\right)^{q^*-2} \geq C'
   \left[\frac{\lambda_n^q \max\{|\yn|,|\ydn|\}^{2-q}}{\alpha} \right]^\frac{1}{q-1},
\end{align*} 
where we used \eqref{anix} in the estimate where $C$ appears and 
in the last step 
a bound for $z>0$ on $I_{\text{SQO}}$ of the form
\[ \hqi{q^*}\left(z^\frac{1}{2-q}\right)^{q^*-2} \geq C' z^\frac{1}{q-1}, \]
that can be obtained by similar means as above. 
\end{proof} 
\begin{remark}
The condition in \eqref{qrrmuckenhouptineqsimple} has an additional sum 
over the index set $I_{HR}'' \cap I_{2}^{c}$ on the left-hand side.  It might be possible to prove 
that this set is empty, e.g., if $I_{HR}'' \subset I_2$. Then the corresponding  sum would vanish, and this happens 
in the linear case ($q=2$). However, we postpone a
more detailed analysis of this issue to the future. 

We also point out that the Muckenhoupt-type conditions 
\eqref{hdmuckenhouptineqsimple},
\eqref{hrmuckenhouptineqsimple}, and 
\eqref{qrrmuckenhouptineqsimple} (except for the additional sum) 
agree with the respective ones 
for the linear case $q=2$ so that they 
appear, in fact, as natural extensions of the linear convergence theory.
\end{remark}

\subsection{Case study of noise restrictions}
For the cases that the operator 
ill-posedness, the regularity of the 
exact solution and the noise show some 
typical behaviour, we investigate 
the restrictions that the 
Muckenhoupt-type conditions 
\eqref{hrmuckenhouptineqsimple} and
\eqref{hdmuckenhouptineqsimple} impose on the 
noise. In particular, we would like to 
point out that the restrictions are not at all
unrealistic and they are satisfied 
in paradigmatic situations. 

Consider a polynomially ill-posed problem, 
\[ \lambda_n = \frac{D_1}{n^{\beta}},  \qquad \beta >0, \] 
where the exact data have a higher 
decay compared to $\lambda_n$  as a result of 
the regularity of the exact solution, 
\[ |\yn| = \frac{D_2}{n^{\nu}},  \qquad \nu >\beta.  \]
We furthermore assume a standard polynomial decay
of the error terms: 
\[ \Dyn = \delta s_n \frac{1}{n^{\kappa}}, \qquad 0 <\kappa <  \nu, \quad s_n \in \{-1,1\}  \] 
The restriction $\kappa < \nu$ is natural 
as the noise is usually less regular than the 
exact solution.  In the linear case, 
Muckenhoupt-type conditions 
lead  to restrictions on the regularity of 
the noise, i.e., upper bounds for 
the decay rate $\kappa$. This is perfectly 
in line with their interpretation as
conditions for sufficiently irregular noise.

In the following we write 
$\sim$ if the left and right expressions 
can be estimated by constants independent of 
$n$. (There might be a $q$-dependence, however).

Let us define the following weight
that appears in \eqref{hrmuckenhouptineqsimple} and
\eqref{hdmuckenhouptineqsimple}:
\begin{align*} 
W_n&:= 	\max\{|\yn|,|\ydn|\}^{2-q}{\lambda_n^q} = 
	\max\{1,|\ydn|/|\yn|\}^{q-2} |\yn|^{2-q}\lambda_n^q \\
&\sim  \frac{1}{ n^{\beta q + \nu(2-q)}}  
\max\{1, |1 + \frac{s_n \delta}{C_2} 
n^{\nu -\kappa}|\}^{2-q}. 
\end{align*} 
We additionally 
impose the restriction that for sufficiently large 
$n$, $W_n\to0$ monotonically. 
If $2-q>0$, this is trivially satisfied, while for 
$2-q<0$, we require that  
\begin{equation}\label{help2}\beta q + \kappa (2-q) > 0, \qquad 
\text{if } 2-q<0. \end{equation}
Under these assumptions, for 
any $\alpha$ sufficiently small, we find an $\ns$ such that 
$W_{\ns} = C_1 \alpha$ and $W_n \leq C_1 \alpha $ for $n\geq \ns$. Expressing $\alpha$ in terms of 
$W_{\ns}$ yields a  
sufficient condition for \eqref{hrmuckenhouptineqsimple},
\eqref{hdmuckenhouptineqsimple} as 
\begin{equation}\label{help1}
    W_{\ns} \sum_{n=1}^\ns 
    \frac{ |\Dyn|^2}{ W_n} \leq C  
    \sum_{n = \ns+1}^\infty   |\Dyn|^2
    \sim \frac{1}{\ns^{2 \kappa -1}}.
\end{equation}
By the straightforward estimate 
$\max\{1, |1 + \frac{s_n \delta}{C_2} 
n^{\nu -\kappa}|\} \sim 
1 + \delta n^{\nu -\kappa}$, 
the inequality \eqref{help1} reduces to 
\begin{equation}
  \frac{(1 + \delta \ns^{\nu -\kappa})^{2-q}}{ \ns^{\beta q + \nu(2-q)-2 \kappa}}   \sum_{n=1}^\ns 
  \frac{n^{\beta q + \nu(2-q)-2 \kappa}}
  {(1 + \delta n^{\nu -\kappa})^{2-q}}
    \leq C \ns.
\end{equation}
For any $x \geq 0$ and $0 \leq z \leq 1$, it holds that 
\[ 1 \leq  \frac{1 + x }{1 +  z x }  \leq \frac{1}{z}. \]
We use this inequality with 
$z = \frac{n}{n^*}$ and $x = \delta \ns$.
Then, we obtain the sufficient conditions 
\begin{align*} 
\begin{dcases}
  \frac{1}{ \ns^{\beta q + \nu(2-q)-2 \kappa
  -  (2-q)(\nu-\kappa)}}   \sum_{n=1}^\ns 
  {n^{\beta q + \nu(2-q)-2 \kappa -
  (2-q)(\nu-\kappa)}}
    \leq C \ns, & 2-q>0, \\
  \frac{1}{ \ns^{\beta q + \nu(2-q)-2 \kappa}}   \sum_{n=1}^\ns 
  {n^{\beta q + \nu(2-q)-2 \kappa}}\leq 
   C \ns,
    & 2 - q<0.
\end{dcases}
\end{align*}
These inequalities are satisfied if the 
exponent for $n$ is strictly larger than $-1$. This finally
leads to the restrictions 
\[ 
\begin{dcases}
\kappa \leq \beta + \frac{1}{q}, & 
q<2, \\ 
\kappa \leq \frac{q}{2}\beta +
\frac{2-q}{q} \nu +\frac{1}{2},
& q>2. 
\end{dcases} \]
Note that for $q>2$, we additionally require 
\eqref{help2}.

We hope to have the
reader convinced that the imposed conditions
on the noise are not too restrictive and, 
in particular, the set of noise that 
satisfies them is nonempty. These 
conditions provide a hint for which cases 
the methods may work or fail:

In case $\frac{3}{2} \leq q \leq 2$, 
both the heuristic discrepancy and the 
Hanke-Raus are reasonable rules. The conditions
on the noise are less restrictive the smaller 
$q$ is. 

In case  $1< q < \frac{3}{2}$, our convergence 
analysis only applies to the heuristic discrepancy rule, 
as the nonnegativity 
condition of the Hanke-Raus rule is not satisfied in this case. It could be said that the 
heuristic discrepancy rule is the more robust one then. 

In case $q>2$, we observe that the restriction 
on the noise depends on the regularity of 
the exact solution. For highly regular exact solutions 
($\nu \gg 1$) the noise condition 
might fail to be satisfied as $q$ becomes very large. 
This happens for both the heuristic discrepancy, 
and the Hanke-Raus rules. 

We did not include the quasi-optimality 
condition in this analysis as this still requires 
further analysis. However, the conditions 
for it are usually even more restrictive than 
for the Hanke-Raus rules and we expect similar 
problems for the case $q>2.$

\section{Numerical experiments}\label{sec:four}
In this section, we would like to numerically illustrate and verify the theoretical findings of the preceding sections. 
It should be stressed that the preceding convergence analysis provides
an important piece in understanding the behaivour, but there exist further 
factors that influence the actual quality of the results using 
heuristic rules (e.g., for optimal-order results, a regularity 
condition on $x^\dagger$ is often required and the value of the constants
in the estimates are important). It should be noted as well that 
our results only proved sufficient conditions for convergence 
and they do not say when a certain method may fail.
(For instance, by including $\delta$-dependent estimates, one may 
find weaker convergence conditions that hold for certain ranges 
of the noise level). Still, a preliminary understanding can 
be gained from this. To further illustrate the behaviour, we 
perform numerical experiments that we describe in this section.

%In particular, we expect the heuristic discrepancy rule to be the most robust, for the Hanke-Raus rule to malfunction for $q$ close to 1 and for problems to arise in case that $q>2$.

In all experiments, we consider $\mathcal{R}=\frac{1}{q}\|\cdot\|^q_{\ell^q}$ in the discretised space $\mathbb{R}^n$. Due to numerical errors resulting from  the discretisation, we opt to choose the parameter $\alpha_\ast\in[\alpha_\text{min},\alpha_\text{max}]$. We also choose to select $\alpha_\text{max}=\|A\|^2$ (apart from for TV regularisation, which we will define) and for the more tricky issue of the lower bound, we set $\alpha_\text{min}=\sigma_{\text{min}}$, the smallest singular value of $A^\ast A$. Other methodologies for selecting $\alpha_\text{min}$ were suggested in \cite{kinderdiscrete,raiksemiheuristic}. Because of several effects,
it may happen that the heuristic functionals exhibit multiple local minima and selecting $\alpha_\ast$ is no longer an obvious task. For this reason, we select $\alpha_\ast$ as the interior global minimum within the aforementioned interval. Additionally, we always rescale the forward operator and exact solution so that $\|A\|=\|x^\dagger\|=1$. In each experiment, we compute 10 different realisations of the noise in which the noise level is logarithmically increasing. We also compute the error (the measure for which will differ for the various regularisation schemes) induced by each parameter choice rule, as well as the optimal parameter choice, which will be computed as the minimiser of the respective error functional itself.

For our operators, we use the tomography operator (\texttt{tomo}) from \emph{Hansen's tools} (cf.~\cite{hansentools}) with $n=625$ and $\texttt{f}=1$. We also define a diagonal matrix $A\in\mathbb{R}^{n\times n}$ with eigenvalues $\lambda_i=C\frac{1}{i^\beta}$, exact solution $x^\dagger=C\cdot s_i\frac{1}{i^\nu}$ and data perturbed by noise $e_i=C\mathcal{N}(1,0)\frac{1}{i^\kappa}$, where $s_i\in\{-1,1\}$ are random and set the parameters as $n=20$, $\beta=4$, $\nu=2$ and $\kappa=1$.

\subsection{{$\ell^1$} regularisation}
A particularly interesting application of convex variational Tikhonov regularisation is the case in which $q=1$, since it is sparsity enforcing. In fact, it is the most sparsity enforcing regularisation method whilst still remaining a convex regularisation problem.
%The choice $q=0$ would, in theory, be more ideal for this purpose, but then $\mathcal{R}$ would no longer be a norm and the minimisation problem becomes concave. 
Significant work in the area of sparse regularisation includes \cite{sparse,sparsityconstraints,ramlau2010sparse}. Whilst it does not fit with the Muckenhoupt-type conditions we derived earlier, it is nevertheless an interesting regularisation scheme for the practitioner who would be eager to see the performance of the studied rules. Note that in this case, we minimise the Tikhonov and Bregman functionals using FISTA (cf.~\cite{beckfista}). The corresponding proximal mapping operator is the soft thresholding operator. Note that in this experiment, we use the tomography operator defined above.

The solution $x^\dagger$ in our customised experiment is chosen to be sparse. For each parameter choice rule, we compute the error as
\[
\operatorname{Err}_{\ell^1}(\alpha_\ast)=\|x^\delta_{\alpha_\ast}-x^\dagger\|_{\ell^1}.
\]
\begin{figure}[H]
	\centering
	\includegraphics[scale=.75]{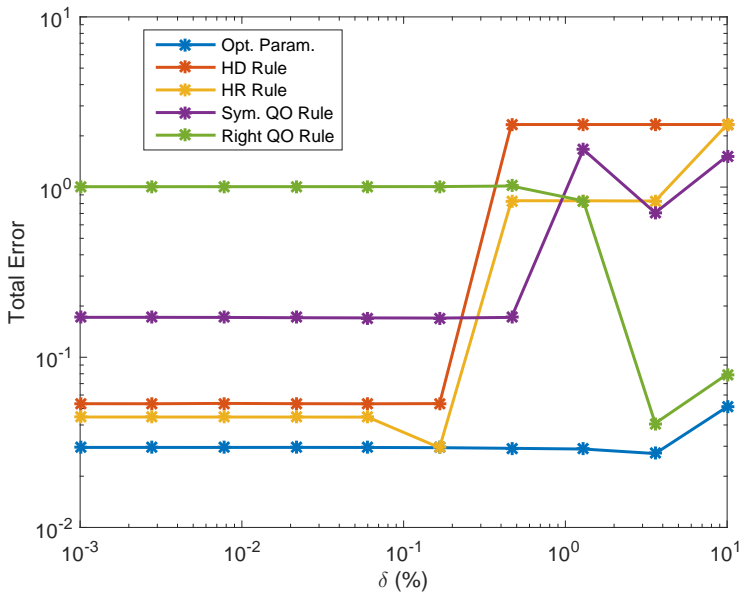}
	\caption{$\ell^1$ regularisation, tomography operator}	\label{l1TotErr}
\end{figure}

We may observe in Figure~\ref{l1TotErr} that, for smaller noise levels, the Hanke-Raus rule appears to be the best performing with the heuristic discrepancy rule performing similarly. The right quasi-optimality rule is particularly subpar for smaller noise levels, whilst for larger noise levels, it appears to be the best performing in fact, whereas the Hanke-Raus and heuristic discrepancy rules take a dip in performance. Indeed, we note that in various other sets of experiments we ran, instances were observed in which the heuristic discrepancy rule slightly trumped the Hanke-Raus rule.

\subsection{$\ell^{\frac{3}{2}}$ regularisation}
An interesting case for the purposes illustrating our theory is when $q=\frac{3}{2}$. Additionally, as with the previous regularisation, we have an analytic formula for the proximal mapping operator corresponding to the regularisation functional. In this scenario, we use the diagonal operator defined above with the given parameters and we compute the error with the Bregman distance; namely,
\[
\operatorname{Err}_{\ell^q}(\alpha_\ast)=D_{\xi}(x^\delta_{\alpha_\ast},x^\dagger),\qquad q\in(1,\infty).
\]
\begin{figure}[H]
	\centering
	\includegraphics[scale=.75]{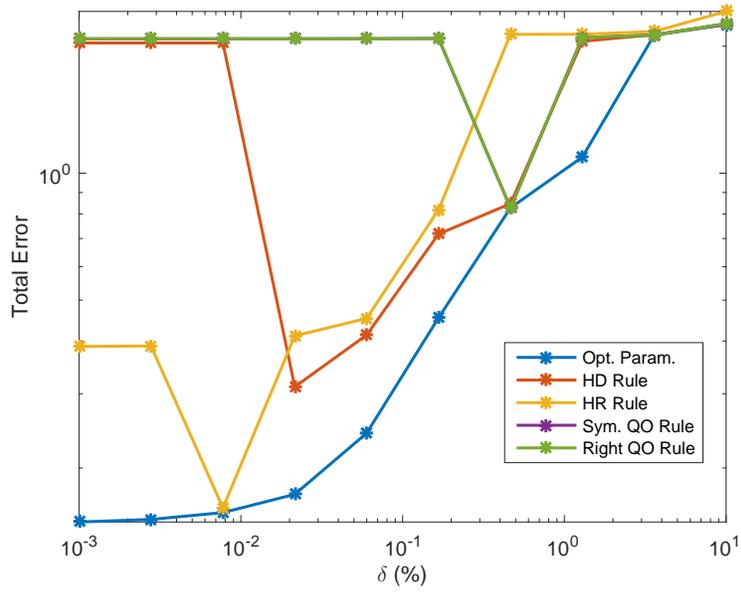}
	\caption{$\ell^{\frac{3}{2}}$ regularisation, diagonal operator}	\label{l32TotErr}
\end{figure}
Observe in Figure~\ref{l32TotErr} that, as in the previous experiment, the Hanke-Raus rule is the best performing one in case that the noise level is relatively small, although for mid-range noise levels, the heuristic discrepancy rule performs slightly better and for larger noise levels still, the quasi-optimality rules match the heuristic discrepancy rule. Note that the quasi-optimality rules appear indistinguishable in this plot and we remark too that the plots of their respective functionals were very similar (see Figure ~\ref{l32Examp}).

\begin{figure}[H]
	\centering
	\includegraphics[scale=.75]{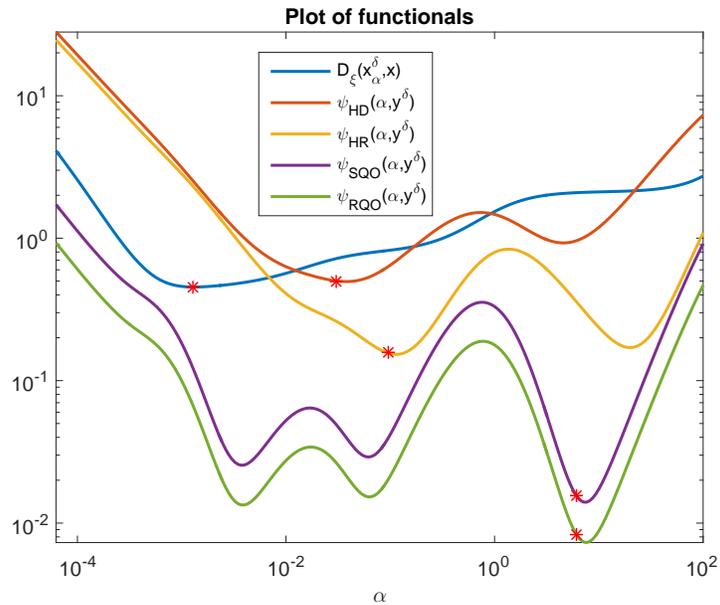}
	\caption{$\ell^{\frac{3}{2}}$ regularisation, plot of functionals}	\label{l32Examp}
\end{figure}
The relatively poor performance of the quasi-optimality rules may be explained by a simple observation of Figure~\ref{l32Examp}. In particular, one may notice that the selected minimisers of the quasi-optimality functionals are suboptimal. Note that they are selected via our procedure to choosing the interior global minimum. Indeed, if the other local minima were selected
(e.g. those left of $\alpha=10^{-2}$), then the results would be much improved. This is a common phenomena in many of our experiments involving the diagonal operator with $q=\frac{3}{2}$. Indeed, we observe that the HD and HR functionals oscillate as well, although in Figure~\ref{l32Examp}, at least, the correct minimisers were chosen.

\subsection{$\ell^3$ regularisation}
Based on the Muckenhoupt-type conditions in the preceding sections, we postulated that for $q>2$, the parameter choice rules we consider are likely to face mishaps. Consequently, we have elected to run a numerical experiment with $q=3$ in order to illustrate what happens in practice. As in the previous experiment, we consider the diagonal operator and compute the error induced by the parameter choice rules as before.

\begin{figure}[H]
    \centering
    \includegraphics[scale=.75]{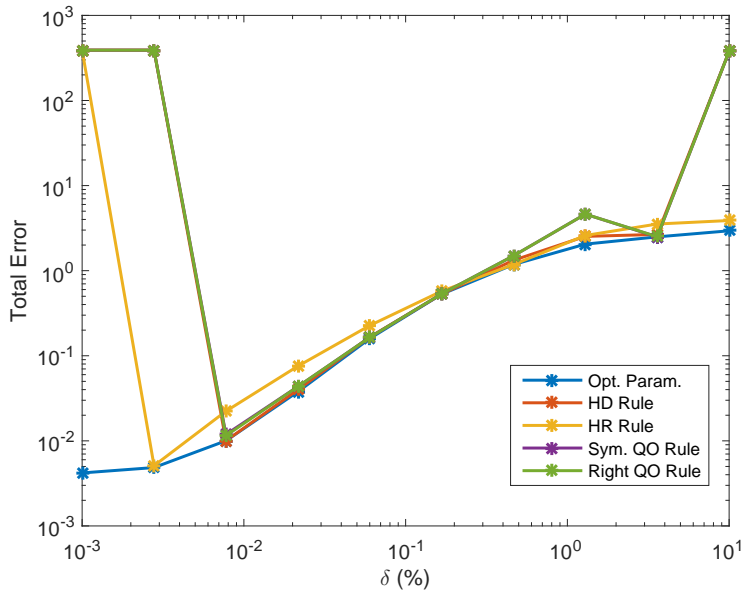}
    \caption{$\ell^3$ regularisation, diagonal operator}
    \label{l3TotErrDiag}
\end{figure}
In Figure~\ref{l3TotErrDiag}, all the rules appear to perform poorly in case the noise level is very small, but perform well overall thereafter, barring the case where the noise level is 10\%, in which case the Hanke-Raus rule is the only one which produces a reasonable error relative to the optimal parameter choice. Note that in the plots of the functionals themselves, we observed that for certain noise levels, some of the functionals did not exhibit reasonable minima.

\subsection{TV regularisation}

Selecting
\[
\mathcal{R}(x):=\sup_{\substack{\phi\in C^\infty_0(\Omega;\mathbb{R}^n)\\\|\phi\|_\infty\le 1}}\int_{\Omega}x(t)\operatorname{div}\phi(t)\,\mathrm{d}t,
\]
with $\text{div}$ denoting the divergence and $\Omega\subset\mathbb{R}^n$ an open subset, yields total variation (TV) regularisation. For the numerical treatment and 
for functions on the real line, 
 this is often discretised as $\mathcal{R}=\sum\|\Delta x\|_{\ell^1}$ with 
 a (e.g. forward) difference operator $\Delta$. For our numerical implementation, we used the FISTA algorithm with the proximal mapping operator for the total variation functional being computed using a fast Newton-type method, courtesy of the code provided by \cite{proxtv1,proxtv2}.
 
 Note that in this case, we choose $\alpha_\text{max}$ such that $\|x^\delta_{\alpha_\text{max}}\|\le C$ for a reasonable constant. Moreover, for each parameter choice rule, we compute the error as
 \[
 \operatorname{Err}_{\text{TV}}(\alpha_\ast)=|\mathcal{R}(x^\delta_{\alpha_\ast})-\mathcal{R}(x^\dagger)|+\frac{1}{\alpha_\ast}\|Ax^\delta_{\alpha_\ast}-y^\delta\|^2,
 \]
which was suggested in \cite{kindermathehofmanntv}. In this instance, we consider the tomography operator.

\begin{figure}[H]
	\centering
	\includegraphics[scale=.75]{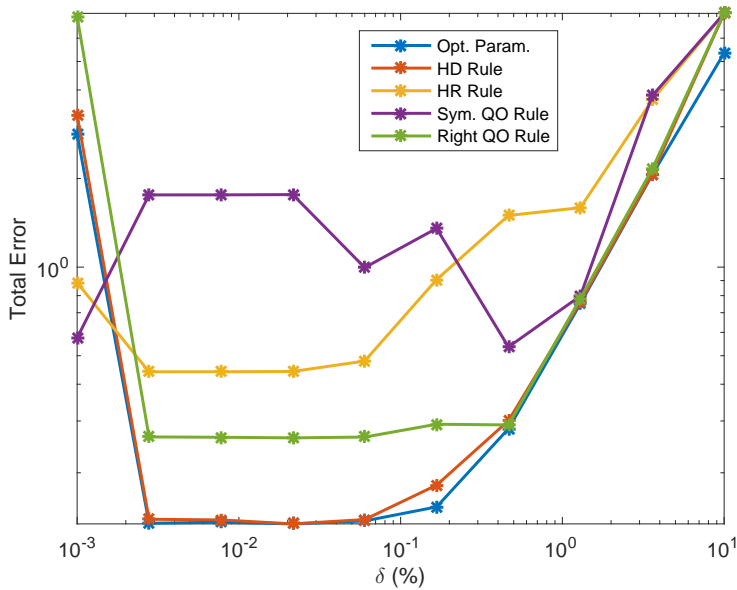}
	\caption{TV regularisation, tomography operator}	\label{TVTotErr}
\end{figure}

One may observe in Figure~\ref{TVTotErr} that the heuristic discrepancy rule appears to overall be the best performing one for all tested noise levels, although for larger noise levels, it is matched and/or trumped by the right quasi-optimality rule. The Hanke-Raus rule, on the other hand, does not appear to present itself as a preferable parameter choice rule in any of the tested noise levels, in this setting at least; that is, unless we compare it with the symmetric quasi-optimality rule which is the worst performing for mid-range noise levels particularly.

\subsection{Summary}
To summarise the numerical experiments presented above, we begin by remarking that the rules worked well, even in instances contrary to the expectations set by the theory. We observed that while none of the studied parameter choice rules were completely immune to mishaps, the heuristic discrepancy rule could perhaps be said to be the most robust overall. Indeed, in light of the above experiments, it is difficult to offer a particular recommendation.

\section{Conclusion}

In conclusion, we introduced four heuristic parameter choice rules for convex Tikhonov regularisation and presented a detailed analysis of the conditions we postulated for when the aforementioned rules are convergent regularisation methods. This involved the more general auto-regularisation conditions, as well as the reduction to the more specific Muckenhoupt-type conditions in case the forward operator is diagonal and the vector spaces are $\ell^q$. Indeed, the analysis for the heuristic discrepancy and Hanke-Raus rules was more in-depth and further investigation of the conditions presented for the symmetric quasi-optimality rule presents room for further research.

We furthermore provided a numerical study to demonstrate the performance of the rules examined in this paper and also illustrate our theoretical findings. Indeed, all the rules in question performed at least reasonably well overall and allow one to conclude that heuristic rules present themselves as viable and in fact, on occasion, attractive options for convex Tikhonov regularisation.

\bibliographystyle{siam}
\bibliography{lib}

\end{document}